\documentclass[reqno,11pt]{amsart}
\usepackage{latexsym}
\usepackage{amsfonts}
\usepackage{amssymb}
\usepackage{mathrsfs}
\usepackage[all]{xy}
\usepackage{bbm}
\usepackage{xcolor}
\usepackage{ulem}
\usepackage[mathscr]{euscript}

\usepackage[english]{babel}

\usepackage{amsmath}
\usepackage{graphicx}
\usepackage[colorinlistoftodos]{todonotes}
\usepackage[colorlinks=true, allcolors=blue]{hyperref}
\usepackage{tikz-cd}
\usepackage{tikz}
\usepackage{pst-node}
\usepackage{mathtools}

\usepackage[T2A]{fontenc}

\theoremstyle{plain}
\newtheorem{thm}{Theorem}[section]
\newtheorem{claim}{Claim}[thm]
\newtheorem{lem}[thm]{Lemma}
\newtheorem{rem}[thm]{Remark}
\newtheorem{prop}[thm]{Proposition}
\newtheorem{cor}[thm]{Corollary}
\newtheorem{defn}[thm]{Definition}
\newtheorem{exmp}[thm]{Example}
\newtheorem{ques}[thm]{Question}

\theoremstyle{definition}

\theoremstyle{remark}

\numberwithin{equation}{section}
\newcommand{\BA}{\mathbf{BA}}
\newcommand{\SL}{\operatorname{SL}}
\newcommand{\GL}{\operatorname{GL}}
\newcommand{\SO}{\operatorname{SO}}
\newcommand{\PSL}{\operatorname{PSL}}
\newcommand{\Ad}{\operatorname{Ad}}
\newcommand{\Lie}{\operatorname{Lie}}
\newcommand{\RP}{\operatorname{Re}}

\newcommand{\area}{\operatorname{area}}
\newcommand{\im}{\operatorname{Im}}
\newcommand\mr{M_{m,n}}
\newcommand\e{\varepsilon}

\newcommand\amr{$A\in M_{m,n}$}
\newcommand\da{Diophantine approximation}

\newcommand{\nz}{\smallsetminus\{0\}}
\newcommand{\Q}{{\mathbb {Q}}}
\newcommand{\R}{{\mathbb{R}}}
\newcommand{\Z}{{\mathbb{Z}}}
\newcommand{\N}{{\mathbb{N}}}
\newcommand{\Hyp}{{\mathbb{H}}}
\newcommand\hd{Hausdorff dimension}
\newcommand{\vv}{{\bf v}}
\newcommand{\vr}{{\bf r}}

\newcommand{\vp}{{\bf p}}
\newcommand{\vq}{{\bf q}}
\newcommand{\x}{{\bf x}}
\newcommand{\ignore}[1]{{}}
\newcommand{\ggm}{G/\Gamma}

\newcommand\eq[2]{
\begin{equation}
\label{eq:#1}
{#2}
\end{equation}
}
\newcommand{\equ}[1]{\eqref{eq:#1}}

\title[Diophantine approximation  in Euclidean norm]{A zero-one law for uniform Diophantine approximation  in Euclidean norm}
\author{Dmitry Kleinbock and Anurag Rao}
\address{Brandeis University, Waltham MA
02454-9110 {\tt kleinboc@brandeis.edu}}
\address{Brandeis University, Waltham MA 02454-9110
{\tt anrg@brandeis.edu}}

\thanks{Research supported by  NSF grants DMS-1600814 and DMS-1900560}
\date{July 23, 2020}
\begin{document}

\begin{abstract}
We study a norm sensitive Diophantine approximation problem arising {from} the work of Davenport and Schmidt on the improvement of Dirichlet's theorem. Its supremum norm case was recently considered by the first-named author and Wadleigh \cite{kw1}, and  here we extend the set-up by replacing the supremum norm with an arbitrary norm.  This gives rise to a class of shrinking target problems for one-parameter diagonal flows on the space of lattices,  with the targets being neighborhoods of the critical locus of the suitably scaled norm ball. We use methods from geometry of numbers to 
generalized a result due to Andersen and Duke \cite{AD} on measure zero and uncountability of the set of numbers (in some cases, matrices) for which Minkowski approximation theorem can be improved. The choice of the Euclidean norm on $\mathbb{R}^2$ corresponds to studying geodesics on a hyperbolic surface which visit a decreasing family of balls. An application of the dynamical Borel-Cantelli lemma of Maucourant \cite{mau} produces, given an approximation function $\psi$, a zero-one law for the set of $\alpha \in \mathbb{R}$ such that   for all large enough $t$  the inequality $\left(\frac{\alpha q -p}{\psi(t)}\right)^2 + \left(\frac{q}{t}\right)^2 < \frac{2}{\sqrt{3}}$ has non-trivial integer solutions.
\end{abstract}

\maketitle

\section{Introduction}\label{intro}

The theory of approximation of real numbers by rational numbers starts with Dirichlet's Theorem (1842): 
\eq{dt}{ \forall\,\alpha\in \mathbb{R}\quad 
\forall \,t> 1\quad \exists\,q \in \N  \text{ with }  \begin{cases}\langle q\alpha \rangle
&\le  \  1/t\\
\quad q
&< \ t\end{cases}\quad;}
here and hereafter $\langle x \rangle$ stands for the distance from $x\in\R$ to a nearest integer. 
\noindent 
{See e.g.\ \cite[Theorem I.I]{Cassels} or \cite[Theorem I.1A]{Schmidt}.} The standard application of \equ{dt} is the following corollary: 
\eq{dc}{ \forall\,\alpha\in \mathbb{R}\quad  
\exists\,\infty\text{  many }q \in \N  \text{ with }  \langle q\alpha \rangle <
1/{q}.}

The two   statements above show two possible ways to pose \da\ problems, often (see e.g.\ \cite{W}) referred to as {\sl uniform} vs.\ {\sl asymptotic}: that is, looking for  solvability of inequalities for all large enough values of   certain parameters  vs.\  for infinitely many (a distinction between limsup and liminf sets). The  rate of approximation given in \equ{dt} and \equ{dc}  works for all $\alpha$, which  serves as a beginning of the {\sl metric theory of \da}, concerned with understanding sets of $\alpha$ satisfying similar conclusions  but with
the right hand sides  replaced by faster decaying functions of $t$ and $q$ respectively.

Those sets are  well studied in the setting of \equ{dc}. Indeed, for a function $\psi: \R_+ \to \R_+$ one considers
$${W}(\psi):= \big\{\alpha\in \R: \exists\,\infty\text{  many }q \in \N  \text{ with }  \langle q\alpha \rangle <
\psi(q)\big\},$$ the set of {\sl $\psi$-approximable} real numbers.
With the notation $\psi_k(t) := 1/t^k$,
\equ{dc} asserts that ${W}(\psi_1) = \mathbb{R}$; moreover {a theorem of Hurwitz (see \cite[1.2F]{Schmidt})} says that 
${W}(c\psi_1) = \mathbb{R}$ for all $c \ge 1/\sqrt{5}$.
Numbers which do not belong to ${W}(c\psi_1) $ for some $c > 0$ 
are called {\sl badly approximable}; we shall denote the set of those numbers by 
$\BA$.
If $\psi$ is non-increasing, Khintchine's Theorem gives the criterion for the Lebesgue measure of ${W}(\psi)$ to be zero or full -- namely, the convergence/divergence of  the series $\sum_k\psi(k)$.

Let us now briefly describe what is known in  the setting of \equ{dt}.   Following \cite{kw1}, for $\psi$ as above say that $\alpha$ is {\sl $\psi$-Dirichlet} if
for all large enough $t$ 
\eq{dtpsi}{ 
\text{there exists $q\in\N$ with }\begin{cases}\langle q\alpha \rangle
&<  \ \psi(t)\\
\quad q
&< \ t\end{cases}\quad.}

Let us denote the set of $\psi$-Dirichlet numbers by $D_{{\infty}}(\psi)$ (the role of the subscript $\infty$ will be clarified below). 
It is immediate\footnote{More precisely, if $\alpha\notin\Q$ (resp.\ if $\alpha\in\Q$), the system \equ{dtpsi} with $\psi = \psi_1$  is solvable for all $t > 1$ (resp.\  for all sufficiently large $t$).} from \equ{dt} 
that $D_{{\infty}}(\psi_1) = \mathbb{R}$.
Also let us say that $\alpha$ is   {\sl Dirichlet-improvable} (see e.g.\ \cite[Definition 5.8]{EW})
if it belongs to $D_{{\infty}}(c\psi_1)$ for some $c<1$. 
Denote by 
$$ {\widehat D_\infty := \bigcup_{c<1}D_{{\infty}}(c\psi_1)}$$
the set of Dirichlet-improvable numbers.
Morimoto (\cite{Mor}, see also \cite[Theorem 1]{Davenport-Schmidt}) was the first to observe that 
the set $\widehat D_\infty$ coincides with  $\Q\cup\BA$, and in particular  has Lebesgue measure zero and is {\sl thick}, that is, intersects  any non-empty open subset  of $\R$  
in a set of full \hd.
%
The latter property, originally established by Jarn\'\i k \cite{J}, was upgraded by Schmidt \cite{S} to   being a {\sl winning set}, and then further strengthened by McMullen \cite{Mc} to {\sl absolute winning}; 
see Remark \ref{haw} for more detail.
 \smallskip
 
Further progress in the study of the sets $D_{{\infty}}(\psi)$ was made in a recent paper \cite{kw1} by the first named author and Wadleigh. Namely, the following was proved: 

{\begin{thm}[\cite{kw1}, Theorems 1.7 and 1.8] \label{dtpsi} Let $\psi$ be a non-increasing function such that $t\psi(t)<1$ for all sufficiently large $t$. Then 
\begin{itemize}
\item[\rm (a)] $D_{{\infty}}(\psi)^c\neq \varnothing$;
\item[\rm (b)] if, in addition, the  function $t\mapsto t\psi(t)$ is non-decreasing,  then the Lebesgue measure of $D_{{\infty}}(\psi)$ (resp.\  of $D_{{\infty}}(\psi)^c$) is zero if \eq{condition}{
{\sum_k
{-\log\big(1-k\psi(k)\big)\left(\tfrac1k-\psi(k)\right)}
}
= \infty \hspace{3mm}(\text{resp.} <\infty).}
\end{itemize}
\end{thm}}

The above theorem was proved via a tight description of elements of $D_{{\infty}}(\psi)$ in terms of their continued fraction expansion. An alternative description can be easily provided via a reduction of the problem to 
dynamics on the space of lattices in $\R^2$. Indeed, 
let $u_\alpha :=  \left[ {\begin{array}{cc}
   1 & \alpha \\
   0 & 1 \\
  \end{array} } \right]$, and consider
  $$\Lambda_\alpha := u_\alpha
\mathbb{Z}^2 = \left\{\left( {\begin{array}{cc}
    \alpha q - p \\
    q \\
  \end{array} } \right): p,q\in\Z\right\}.$$ Denote by $B(r)$ the open ball in $\R^2$ of radius $r$ centered at $0$ with respect to the supremum norm on $\R^2$. Then it is easy to see that $\alpha\in D_{{\infty}}(\psi)$ if and only if for all large enough $t$ the lattice $\Lambda_\alpha$ has a nonzero vector inside the rectangular box $\left[ {\begin{array}{cc}
   \psi(t) & 0 \\
   0 & t \\
  \end{array} } \right]
  B(1)$.
  \smallskip

While the use of supremum norm arises naturally from the problem considered by Dirichlet, it seems natural to state similar problems for an arbitrary norm $\nu$, thereby replacing balls $B(r)$ with $$B_\nu(r):= \{\x\in\R^2: \nu(\x) < r\},$$
open balls centered at $0$ with respect to the norm $\nu$. A recent article \cite{AD} by Andersen and Duke provides evidence that this norm sensitive approximation problem was studied by Hermite only a few years after the work of Dirichlet, and later on by Minkowski. \ignore{Our interest in the norm change stems from the \textit{shrinking target problem} that arises when Dirichlet's theorem is restated in the language of flows on homogeneous spaces (cf.\ \S\ref{dyn}). As we will see, the shapes of the compact targets that arise, depending crucially on the norm, make a big impact on the complexity of the problem. This is in contrast with the shrinking target problem arising from Khintchine's theorem where the targets are non-compact cusp neighborhoods 
in the homogeneous space.
\smallskip
}

Keeping up with the notation in 
\cite{AD}, we define for each norm $\nu$ a critical value \eq{crit}{\Delta_{\nu} :=  \text{ the smallest co-volume over all lattices intersecting $B_{\nu}(1)$ trivially.}
}
 Much is known about these constants and the set of lattices that attain this lower bound. For example, $\Delta_{\nu}$ is computed as the minimal area of a parallelogram with one vertex at the origin and the other three on the boundary of $B_{\nu}(1)$. 

This critical value is used in our generalization of $D_{{\infty}}(\psi)$.
Namely, let us say that $\alpha$ is {\sl $(\psi,\nu)$-Dirichlet}, or    $\alpha\in D_\nu(\psi)$, if  \eq{psinu}{\Lambda_\alpha \cap\left[ {\begin{array}{cc}
   \psi(t) & 0 \\
   0 & t \\
  \end{array} } \right]
  B_\nu\left(\frac{1}{\sqrt{\Delta_{\nu}}}\right) \ne \{0\}\text{ whenever $t > 0$ is large enough}.} 
Note, this definition is consistent with  {what we had} before since the critical value for the supremum norm is $1$. {Furthermore, it is easy to see that the function $\psi_1$ again plays the role of a critical parameter: if $c>1$, then $D_{\nu}\left(c\psi_1 \right) =\R$.}

{We will always assume $\psi$}  to be non-decreasing  {and continuous}. Note  that the case $\nu = \|\cdot\|_\infty$ has an extra feature:  if \equ{psinu} is true for all large enough $t\in\N$, then the same is true for all large enough $t>0$. This makes it possible to reduce the problem to continuous functions $\psi$. This argument does not  apply to the set-up of arbitrary norms $\nu$. 
However,    {for the most part  the scope of our paper will allow us to only deal with the continuous case, see Remark \ref{cont}}.

  {When $\nu(\x) = \|\x\|_p$, the $\ell^p$ norm, we shall denote $B_\nu(r)$ by $B_p(r)$, $D_\nu(\psi)$ by $D_p(\psi)$ and $\Delta_{\nu}$ by $\Delta_p$; the set-up discussed in \equ{dtpsi} corresponds to $p = \infty$.} 
 In the case of  the Euclidean norm, which will be the main topic of this paper, 
 ${D_{2}(\psi)}$ is the set of  $\alpha\in\R$ for which the inequality  
$$\left(\frac{\langle\alpha q \rangle}{\psi(t)}\right)^2 + \left(\frac{q}{t}\right)^2 < \frac{2}{\sqrt{3}} $$
is solvable in $q\in \N$ for all large enough $t$. (Note that $\Delta_2=\sqrt{3}/2$.)
\smallskip

\ignore{Let $X=\SL_2(\R)/\SL_2(\Z)$, the space of unimodular lattices in $\R^2$. For the reduction to dynamics, we proceed from \equ{psinu} to get 
\eq{first-dynamical-restatement}{\alpha\in D_\nu(\psi) \iff b_t\Lambda_{\alpha}\notin {\mathcal K}_\nu\left( \sqrt{t\psi(t)}\right) \text{ for all large enough } t}
where \eq{first-bt}{b_t :=\left[ \begin{array}{cc}
   \sqrt{t/\psi(t)} & 0 \\
   0 & \sqrt{\psi(t)/t} \\
  \end{array}\right] \text{ and } {\mathcal K}_\nu(r):= \left\lbrace \Lambda \in X : \Lambda \cap B_{\nu}\left(r/\sqrt{\Delta_{\nu}}\right) = \lbrace 0 \rbrace \right\rbrace.}
From this it is easy to see that the function $\psi_1(t)=\frac{1}{t}$ plays the role of a critical parameter. If $c>1$, $D_{\nu}\left(c\psi_1 \right) =\R$ and if $c<1$,  $D_{\nu}\left(c\psi_1 \right) $ has measure zero (c.f. proposition \ref{erg}).}

In the paper \cite{AD}, Andersen and Duke obtained several results under the additional assumption that $\nu$ is {\sl strongly symmetric}, that is satisfies
$$
\nu\big((x_1,x_2)\big) = \nu\big((|x_1|,|x_2|)\big) \text{ for all }
(x_1,x_2)\in\R^2
.
$$
In particular, they considered a generalization of the set $\widehat D_\infty$:
\eq{dionedim}{\widehat D_\nu := \bigcup_{c < 1}D_\nu(c\psi_1),}
which they referred to as `the set of numbers for which {Minkowski's approximation theorem can be improved}', 
and proved
\begin{thm}[\cite{AD}, Theorem 1.1] \label{dtad} For any strongly symmetric norm $\nu$ on $\R^2$, the set $\widehat D_\nu$
{\rm (a)}  has Lebesgue measure zero,
and {\rm (b)} is uncountable.
 \end{thm}

\noindent{In the present paper we would like to take an arbitrary norm $\nu$ on $\R^2$ and} consider the following

\medskip

\noindent{{\bf Questions.}
\begin{itemize}
\item[\rm (i)] Will part (a)  of the above theorem hold in that generality? 
\item[\rm (ii)] 
Will part (b)  hold, and can one strengthen its conclusion 
 by showing that  the set $\widehat D_\nu$
 is 
thick? 
winning? absolute winning?
\item[\rm (iii)] 
Is it true that  $D_\nu(\psi_1) = \R$?
\item[\rm (iv)] It is true that $D_\nu(\psi)^c\ne\varnothing$ 
whenever {$\psi(t)<\psi_1(t)$} for all sufficiently large $t$?
\item[\rm (v)] Perhaps under some  additional condition such as the monotonicity of  the function $t\mapsto t\psi(t)$, can one find a criterion for the Lebesgue measure of $D_\nu(\psi)$ to be zero or full?
\end{itemize}}

\ignore{\begin{ques}\label{precise}What can one say about {the critical function $\psi_1(t)$ itself? i.e.\ is it true that $D_\nu(\psi_1) = \R$?}\end{ques}

\begin{ques}\label{nonempty}It is true that $D_\nu(\psi)^c$ is non-empty whenever {$\psi(t)<\psi_1(t)$} for all sufficiently large $t$? 
\end{ques}

\begin{ques}\label{zerofull} Perhaps under some  additional condition such as the monotonicity of  the function $t\mapsto t\psi(t)$, can one find a criterion for the Lebesgue measure of $D_\nu(\psi)$ to be zero or full?
\end{ques}}

We 
answer Question (i) affirmatively in the very general set-up of systems of $m$ linear forms in $n$ variables, where $m,n\in\N$ are arbitrary (Theorem~\ref{measurezero}). 
With regards to Question (ii),  strengthening Theorem 1.2(b),  we prove
\begin{thm} \label{winning} {For any norm $\nu$ on $\R^2$, the set $\widehat D_\nu$ is absolute winning.}
\end{thm}

We also obtain some partial results in the higher-dimensional case; in particular, a   modification of the absolute winning property, namely {\sl hyperplane absolute winning (HAW)} introduced in \cite{BFKRW},  will be shown to hold for a multi-dimensional analogue of the set $\widehat D_\nu$  where $\nu$ is the Euclidean norm on $\R^{m+n}$ (see Theorem \ref{euclidean}).

\ignore{Suppose that $\nu = \|\cdot\|_p$, where 
$1\le p \le \infty$, is the $\ell^p$ norm on $\R^2$. Then 
the set  $\widehat D_\nu$
is thick.}


\medskip
 For the rest of the questions we restrict our attention to
the Euclidean norm on $\R^2$. Specifically, we prove the following theorems:

\begin{thm} \label{dtpsieuclcritical} {$D_2\big(\psi_1\big) = \R$; in other words, for any $\alpha\in\R$ the inequality $$\langle\alpha q \rangle < \frac{1}{t}\sqrt{\frac{2}{\sqrt 3}-\left(\frac{q}{t}\right)^2} 
$$ 
is solvable for all large enough $t > 0$.}
\end{thm}

{\begin{thm} \label{dtpsieucl1} 
 Let $\psi$ be a non-increasing continuous function
such that \eq{smaller}{\psi(t)< \psi_1(t)=\frac{1}{t}
\quad\text{for all sufficiently large }t.} 
Then $D_2(\psi)^c \neq \varnothing$. \end{thm}
\begin{thm} \label{dtpsieucl2}  Let $\psi$ be  as in Theorem \ref{dtpsieucl1}, and assume, {in addition,} that
\eq{monotone}{\text{the 
function $t\mapsto t\psi(t)$ is non-decreasing.}} Then  the Lebesgue measure of $D_2(\psi)$ (resp.\  of $D_2(\psi)^c$) is zero whenever \eq{conditioneucl}{
{\sum_k \big(\psi_1(k)-\psi(k)\big) }
\ignore{= \sum_k\left(\frac{1}{k} - \psi(k)\right)}
= \infty \hspace{3mm}(\text{resp.} <\infty).}
\end{thm}}

\noindent Note the difference between \equ{conditioneucl} and \equ{condition}: the latter can be written as 
$$
{\sum_k\big(\ignore{r_\infty^2}\psi_1(k) - \psi(k)\big)}
\log\left(\frac1{k\big(\ignore{r_\infty^2}\psi_1(k) - \psi(k)\big)}\right)
= \infty \hspace{3mm}(\text{resp.} <\infty);$$
that is, compared with  \equ{conditioneucl}, has an extra logarithmic term. 
We have the following 
examples\footnote{These functions $\psi$ are  only decreasing for large enough values of $t$ --  but clearly only the eventual behavior of $\psi$ is relevant to the problem.} demonstrating condition \equ{conditioneucl}:
\begin{itemize}
    \item if $\psi(t)= \frac{1}{t}-\frac{1}{t^{k+1}}$, then $D_2(\psi)$ has full measure when $k>0$;
    \item if $\psi(t)= \frac{1}{t}-\frac{1}{t(\log t)^k}$, then $D_2(\psi)$ is null for $k\leq 1$ and conull for $k>1$.
\end{itemize}

\begin{rem}\label{cont} {\rm One can notice that condition \equ{monotone} of Theorem \ref{dtpsieucl2}, together with the assumption that $\psi$ is non-increasing, forces $\psi$ to be continuous. On the other hand, 
Theorem \ref{dtpsieucl1} would clearly hold for discontinuous functions as long as \equ{smaller} is replaced by
$$\inf_{t_0<t<t_1} \left(\ignore{r_2^2}\psi_1(t)-\psi(t)\right) >0\quad\text{for all sufficiently large }t_0\text{ and all }t_1>t_0.$$}
\end{rem}

\smallskip
This article is structured as follows. In \S\ref{dyn}  we generalize the   problems described above to the set-up of systems of $m$ linear forms in $n$ variables, and describe the connection with
  diagonal flows on the space of lattices. 
In this generality, i.e.\ for arbitrary $m$ and $n$, in \S\ref{ergodicity} we address Questions (i) and (ii) from the above list. 
The first one is answered for an arbitrary norm  in Theorem~\ref{measurezero}. 
For the second one a sufficient condition for the 
HAW property of the higher-dimensional analogue of the set \equ{dionedim} is deduced from a recent work by An, Guan and the first-named author \cite{AGK}. We use that condition to answer  Question (ii) 
for the Euclidean norm on $\R^{m+n}$ and for arbitrary norm on $\R^2$.
  {Theorems \ref{dtpsieuclcritical} and 
  \ref{dtpsieucl1}  are proved in \S\ref{h2} by a geometric argument dealing with geodesics in the upper-half plane. In \S\ref{transition} we show how to deduce Theorem \ref{dtpsieucl2} from a corresponding dynamical zero-one law 
 for geodesic flows on finite volume hyperbolic surfaces due to Maucourant \cite{mau}, which is discussed in detail in  \S\ref{sectionmau}.}

\subsection*{Acknowledgements}
The authors would like to thank Jinpeng An, Nikolay Moshchevitin, Lam Pham, Srini Sathiamurthy, Nick Wadleigh and Shucheng Yu  for helpful discussions, and the anonymous reviewers for a careful reading of the paper which has led to multiple improvements.

\section{Systems of linear forms and reduction to dynamics}\label{dyn}


In this section we generalize the notion of $(\psi,\nu)$-Dirichlet real numbers to the set-up of systems of linear forms.  
Fix positive integers $m,n$, put $d = m+n$, and denote by $\mr$ 
the space of $m\times n$ matrices with real entries, interpreted as systems of $m$ linear forms in $n$ variables, $\x\mapsto A\x$.
Let $\nu$ be an arbitrary norm on $\R^{d}$, and
let $\psi$ be a non-negative function defined on an interval $[t_0,\infty)$ for some $t_0 \ge 1$. 
Generalizing  \equ{psinu}, let us say that  $\mr$\ is {\sl $(\psi,\nu)$-Dirichlet}, 
and write $A\in {D}_\nu(\psi)$, 
if 
for every sufficiently
large $t>0$  one can find
$\vq  \in \Z^n\nz$ and $\vp \in \Z^m$
such that the vector $\left( {\begin{array}{cc}
    A\vq - \vp \\
    \vq \\
  \end{array} } \right)$ is inside the ``generalized ellipsoid'' $\left[ {\begin{array}{cc}
   \psi(t)I_m & 0 \\
   0 & tI_n \\
  \end{array} } \right]
  B_\nu\left(\frac{1}{{\Delta_{\nu}}^{{1/d}}}\right)$, where  
  $\Delta_{\nu}$ is 
  as in \equ{crit}.
  {Here} $I_k$ stands for the $k\times k$ identity matrix, and,  as before, we use notation $$B_\nu(r):= \{\x\in\R^{d}: \nu(\x) < r\}.$$

When $\nu = \|\cdot\|_\infty$ is the supremum norm on $\R^{d}$, we recover the standard set-up of uniform simultaneous \da: indeed,  in that case  the condition  $A\in {D}_\nu(\psi)$ is equivalent to the  system\footnote{This definition differs slightly from the one used in \cite[\S4]{kw1} and \cite[Definition 2.2]{CGGMS}, where  the system 
$\begin{cases}
\|A\vq - \vp \|_\infty^m 
&< \psi(t)\\
\qquad\ \ \|\vq  \|_\infty^n 
&< \ t\end{cases}$
was used in place of    \equ{system}.}  \eq{system}{\begin{cases}
\|A\vq - \vp \|_\infty 
&< \psi(t)\\
\qquad\ \ \|\vq  \|_\infty 
&< \ t\end{cases}}
having a nonzero solution $(\vp,\vq)$ for all large enough $t$.


Let us now restate the $(\psi,\nu)$-Dirichlet property in the language of dynamics on the space $X = \SL_{d}(\R)/\SL_{d}(\Z)$ of unimodular lattices in $\R^{d}$. Define
 $$
 u_A := \left[ {\begin{array}{cc}
   I_m & A \\
   0 & I_n \\
  \end{array} } \right]
  $$
  and
  $$
  \Lambda_A := \left\{\left( {\begin{array}{cc}
    A\vq - \vp \\
    \vq \\
  \end{array} } \right): \vp\in\Z^m,\ \vq\in\Z^n\right\} = 
u_A\Z^{d}\in X;
  $$
 then $A\in {D}_\nu(\psi)$ if and only
   if
  \eq{psinumplusn}{\Lambda_A \cap\left[ {\begin{array}{cc}
   \psi(t)I_m & 0 \\
   0 & tI_n \\
  \end{array} } \right]
  B_\nu\left(\frac{1}{{\Delta_{\nu}}^{1/d}}\right) \ne \{0\}
  } 
whenever $t > 0$ is large enough. 
Note that the determinant of $\left[ {\begin{array}{cc}
   \psi(t)I_m & 0 \\
   0 & tI_n \\
  \end{array} } \right]$ is equal to $\psi(t)^mt^n$; thus, to reduce the problem to 
  {the $\SL_{d}(\R)$-action} on $X$, one can 
  introduce the matrix
  \eq{btmn}{b_t :=\left[ {\begin{array}{cc}
  \left(\frac t{\psi(t)}\right)^{n/{d}} I_m & 0 \\
   0 &  \left(\frac {\psi(t)}t\right)^{m/{d}} I_n\\
  \end{array} } \right]\in\SL_{d}(\R).
  }
 Then \equ{psinumplusn} becomes equivalent to 
$$ {b_t\Lambda_A\notin {\mathcal K}_\nu\left( t^{n/{d}}\psi(t)^{m/{d}}\,\right),}$$ 
where 
$${{\mathcal K}_\nu(r) {:=} \left\lbrace \Lambda \in X : \Lambda \cap B_{\nu}\left(\frac{{r}}{{\Delta_{\nu}}^{1/d}}\right) = \{0\}\right\rbrace. }$$
Note that for any norm $\nu$, any $r > 0$ and in any dimension the  sets ${\mathcal K}_\nu(r)$ are compact in view of Mahler's Compactness Criterion \cite{ma0}. 

The use of $b_t$ as in  \equ{btmn}
has two obvious disadvantages: it is not a group parametrization, and its definition depends on the choice of the function $\psi$. It is much more natural to use a group parametrization: 
\eq{defas}{F :=\{a_s : s\in\R\},\quad\text{where }a_s: = \left[ {\begin{array}{cc}
   e^{s/m}I_m & 0 \\
   0 & e^{-s/n}I_n \\
  \end{array} } \right]. 
  }
 This can be achieved by the change of variables
\eq{reparam}{s:= \frac{mn}{d}\ln\left(\frac t{\psi(t)}\right),} 
and {then, using the monotonicity and continuity\footnote{{The continuity of $\psi$} is needed to uniquely define $r$ in terms of $\psi$ via \equ{reparam}. As was noted in the introduction, $\psi$ can  be assumed to be continuous in the supremum norm case. To deal with arbitrary norms, specifically the Euclidean norm,  the  continuity assumption needs to be added. As mentioned in Remark \ref{cont}, the scope of Theorems \ref{dtpsieucl1} and \ref{dtpsieucl2} allows one to do this without loss of generality.}} of $\psi$, to 
  define a function $r:[s_0,\infty)\to \R_+$, where $s_0 := \frac{mn}{d}\ln\left(\frac{t_0}{\psi(t_0)}\right)$, by the equation \eq{dani}{r \left(\tfrac{mn}{d}\ln\big(\tfrac t{\psi(t)}\big)\right) = t^{n/{d}}\psi(t)^{m/d}. }
The passage from $\psi$ to $r$ and back is usually referred to as the {\sl Dani correspondence}. (See \cite[Proposition 4.5]{kw1} where it is written down for the supremum norm.) We have arrived at the following 

\begin{prop}\label{danicorr} Let $\nu$ be an arbitrary norm on $\R^{d}$,  let $\psi$ be a non-increasing continuous function, and let $r(\cdot)$ be the unique function related to $\psi$ via \equ{dani}. Then ${A}\in D_\nu(\psi)$  if and only if 
 \eq{psinudani}{a_s\Lambda_A\notin {\mathcal K}_\nu\big(r(s)\big) \text{ whenever $s$ is large enough}.}
\end{prop}


Observe that when $\psi(t) = c\psi_{n/m}(t) = c t^{-  n/m}$,  one has  
$ t^{n/{d}}\psi(t)^{m/{d}}  \equiv c^{m/{d}}$;
in other words, under the Dani correspondence  \eq{corrconst}{\text{$\psi = c\psi_{n/m}$ corresponds to the constant function }r(s) \equiv c^{m/{d}}.}
By definition of $\Delta_{\nu}$, ${\mathcal K}_\nu(r) = \varnothing$ for ${r} > 1$, which immediately implies that  \eq{generaldt}{D_\nu(c\psi_{n/m}) = M_{m,n}\quad\text{whenever }c > 1.}
{Note that} when $\nu = \|\cdot\|_\infty$, the critical {value} $\Delta_\infty$ is equal to $1$ in any dimension, and  \equ{generaldt} corresponds to the classical Dirichlet's Theorem for simultaneous approximation.

{Note also that} when $r<1$, ${\mathcal K}_\nu(r)$ is {a} non-empty, compact set containing an open neighborhood of 
$${\mathcal K}_\nu(1) = \bigcap_{r<1}{\mathcal K}_\nu(r).$$
The latter set, called the {\sl critical locus} corresponding to the norm $\nu$,  plays an important role for the problems we are considering; elements of this set are called {\sl critical lattices}. 

Another way of defining the set ${\mathcal K}_\nu(1)$ is through  the following function on $X$:
\eq{delta}{\delta_\nu({\Lambda}):={\Delta_\nu}^{1/{d}} \inf\limits_{\x\in {\Lambda} \nz } \nu(\x),}
that is,  $\delta_\nu({\Lambda})$ is the suitably normalized length of a shortest nonzero vector of $\Lambda$.
Clearly $\delta_\nu$ is continuous, and we have the equality ${\mathcal K_\nu}(r)=\delta_\nu^{-1}\big([r,1]\big)$; in particular, the critical locus ${\mathcal K}_\nu(1) = \delta_\nu^{-1}(1)$  
consists of all lattices maximizing $\delta_\nu$, the value of the maximum being equal to $1$ due to our normalization.

{When $\nu = {\|\cdot\|_\infty}$ is the supremum norm, the structure of its critical locus ${\mathcal K_\infty}(1)$ is described by the Haj\'os-Minkowski Theorem (see \cite[\S IX.1.3]{ca} {and also \cite[\S 3.3]{F} for details of the proof}).
In particular for $d = 2$ one has
\begin{equation}\label{critlocus} {\mathcal K_\infty}(1) = 
    \left\lbrace \left[{\begin{array}{cc}
   1 & \alpha \\
   0 & 1 \\
  \end{array}}\right]  \Z^2: \alpha\in \R\right\rbrace \bigcup  \left\lbrace \left[{\begin{array}{cc}
   1 & 0 \\
   \alpha & 1 \\
  \end{array}}\right] \Z^2: \alpha\in \R\right\rbrace.
\end{equation}

Something can also be said for the case of the Euclidean norm on $\R^d$ for arbitrary $d$ (see Theorem \ref{euclidean}  below).
In general, however, each norm comes with its own peculiarities, with difficulty increasing with dimension.
In two dimensions an extensive theoretical study of critical loci appears in the papers \cite{ma1,ma2} of Mahler. 
{For example, when $d=2$ and $\nu$ a norm with polygonal unit ball or an $\ell^p$ norm,  the critical locus $\mathcal{K}_{\nu}(1)$  is finite    \cite{ma3, Glazunov-Golovanov-Malyshev}.
See also \cite{KRS} for examples of critical loci of more complicated nature, e.g.\ of fractional Hausdorff dimension.}
In higher dimensions one can find in \cite[Chapter V]{ca} many useful necessary conditions for a lattice to be critical.}


Using  \equ{corrconst} and Proposition \ref{danicorr} we can immediately derive a dynamical description for the higher-dimensional analogue of the set \equ{dionedim}, that is, the set 
 $${\widehat D_\nu := \bigcup_{c < 1}
D_\nu(c\psi_{n/m})}
$$
of {\sl $\nu$-Dirichlet-improvable} systems of linear forms  \amr:
\begin{prop}\label{exceptional} 
$A\in \widehat D_\nu$ if and only if there exists an open neighborhood $U$ of ${\mathcal K}_\nu(1)$ 
such that $a_{s}\Lambda_A\notin U$ for large enough $s$.
\end{prop}

\section{$\widehat D_\nu$ is a winning set of measure zero}\label{ergodicity}
One implication of the correspondence described in the preceding section is an affirmative action to Question (i)  from the introduction: 
\begin{thm}\label{measurezero} For any  norm $\nu$ on $\R^{d}$, the set $\widehat D_\nu$   has Lebesgue measure zero.
\end{thm}

In view of Proposition \ref{exceptional}, it is clear that the above theorem immediately follows from
\begin{prop}\label{erg} For Lebesgue-almost every \amr\ the trajectory \eq{traj_0}{\{a_s\Lambda_A: s > 0\}} is dense in $X$.
\end{prop}


For the case $\min(m,n) = 1$ the proof {of Proposition \ref{erg}}, which capitalizes on \cite{Davenport-Schmidt2} and is based on geometry of numbers, can be found  in  \cite{Schmidt-density}. In \cite{Davenport-Schmidt2}  a slightly weaker statement was used to establish   {Theorem} \ref{measurezero} for $\nu = \|\cdot\|_\infty$.
It is not clear if the argument of \cite{Schmidt-density} extends to arbitrary $m,n$.
However, as first observed by Dani, the above proposition can be easily derived from the ergodicity of the  $a_s$-action on $X$.
  
\begin{proof}[Proof of Proposition \ref{erg}] The argument is fairly standard. We need to prove that for Lebesgue-a.e.\ \amr\ the trajectory $\{a_s\Lambda_A: s > 0\}$ is dense in $X$. It is easy to see that  {elements $g\in\SL_{d}(\R)$ of the form}
\eq{decomp}{g = \left[ {\begin{array}{cc}
   I_m & 0 \\
   C & I_n \\
  \end{array} } \right]  \left[ {\begin{array}{cc}
   B & 0 \\
   0 & D \\
  \end{array} } \right]
{u_A},
}
  where \amr, {$C\in M_{n, m}$, $B\in M_{m, m}$, $D \in M_{n, n}$} with $\det(B) \det(D) = 1$, form an open dense subset of $\SL_{d}(\R)$ of full Haar measure.  That is, $\SL_{d}(\R)$ is locally a direct product of \eq{defh}{H_- := \left\{\left[ {\begin{array}{cc}
   I_m & 0 \\
   C & I_n \\
  \end{array} } \right]: C\in M_{n, m}\right\},\ H := \left\{
u_A: A\in \mr
  \right\}} (those are the contracting and expanding horospherical subgroups\footnote{A subgroup $H$ of a Lie group $G$ is said to be {\sl expanding horospherical}  relative to a one-parameter semigroup $\{a_s : s > 0\}\subset G$ if its Lie algebra is a direct sum of generalized eigenspaces of $\Ad(a_1)$ with eigenvalues of absolute values bigger than $1$; the contracting  horospherical subgroup is defined similarly.} relative to $\{a_s : s > 0\}$), and the centralizer $$Z = \left\{\left[ {\begin{array}{cc}
    B & 0 \\
   0 & D \\
  \end{array} } \right]:B\in M_{m, m},\ D \in M_{n, n}\right\}$$ of $a_s$.   
  On the other hand, the ergodicity of the $a_s$-action on $X$ (Moore's Ergodicity Theorem) implies that  for Haar-a.e.\ $g\in\SL_{d}(\R)$ the trajectory $\{a_sg\Z^{d}: s > 0\}$ is dense in $X$. Now one can write
  $$a_sg\Z^{d} = \left[ {\begin{array}{cc}
   I_m & 0 \\
   e^{-\frac{d}{mn}s}C & I_n \\
  \end{array} } \right]  \left[ {\begin{array}{cc}
   B & 0 \\
   0 & D \\
  \end{array} } \right]a_s\Lambda_A,
$$
  and, 
since $ \left[ {\begin{array}{cc}
   I_m & 0 \\
   e^{-\frac{d}{mn}s}C & I_n \\
  \end{array} } \right] $ tends to $I_d$ as $s\to\infty$, conclude that,
 for $g$ of the form \equ{decomp}, 
 $\{a_sg\Z^{d}: s > 0\}$ is dense in $X$ if and only if so is $\{a_s\Lambda_A: s > 0\}$. The claim then follows from Fubini's Theorem and the local product structure of Haar measure on $\SL_{d}(\R)$. 
  \end{proof}

Let us now   address Question (ii) from the introduction in the bigger generality of systems of linear forms, that is, construct sufficiently many 
$\nu$-Dirichlet-improvable \amr. 
In view of Proposition \ref{exceptional} the problem can be restated as follows: find sufficiently many
 \amr\ such that the set of limit points of the trajectory  \equ{traj_0} is disjoint from ${\mathcal K}_\nu(1)$. 
This circle of problems has a rich history, see \cite{K99,  AGK} and references therein.
In order to use some results from the aforementioned papers we need to introduce some more terminology.

\begin{defn}\label{trans} \rm
Let $G$ be a Lie group, $\Gamma$ a discrete subgroup, $Z$  a ${\mathcal C}^1$ submanifold of $\ggm$,  and  let $F$ and $H$ be two closed subgroups of $G$.
We use $T_x(\cdot)$ to denote the tangent space to a manifold at a point $x$, and $\operatorname{Lie}(\cdot)$ to denote the Lie algebra of a group, i.e.\ the tangent space at the identity element of the group.
We will say that  $Z$   is {\sl $(F,H)$-transversal at $x\in Z$} if the following holds: 
\smallskip

\begin{itemize}
\item[(i)] $T_x(Fx)\not\subset T_xZ$;
\smallskip

\item[(ii)] $T_x(Hx)\not\subset T_xZ\oplus T_x(Fx)$.
 \end{itemize}
 \smallskip

We will say that $Z$   is {\sl $(F,H)$-transversal} if it is $(F,H)$-transversal at its every point. This is a simplified version of the terminology introduced in \cite[\S 4]{K99}. Note that in a special case when $Z$ is an orbit of a Lie subgroup $L$ of $G$, the above conditions (i), (ii) can be easily restated as \eq{translie1}{\Lie(F)\not\subset \Lie(L)}
and 
 \eq{translie2}{ \Lie(H)\not\subset \Lie(L)\oplus\Lie(F)}
 respectively.
\end{defn}
The following was proved in  \cite{K99} for arbitrary $G$ and $\Gamma$, see  \cite[Corollary 4.3.2]{K99}: if $F = \{a_s : s\in\R\}$ is a non-quasiunipotent\footnote{i.e.\ $\Ad(a_1)$ has an eigenvalue with absolute value different from $1$} one-parameter subgroup of $G$,  and $H$ is the expanding horospherical subgroup  relative to $\{a_s : s > 0\}$, then for any  ${\mathcal C}^1$ compact $(F,H)$-transversal submanifold $Z$ of $\ggm$ and any $x\in X$, the set 
\eq{efz}{
\left\{h\in H: \overline{\{a_shx: s \ge 0\}}\cap Z = \varnothing\right\}
}
is thick.
This has been  strengthened in a recent work of An, Guan and the first-named author \cite{AGK}.  Namely, they introduced a notion of {\sl maximally  expanding horospherical subgroup} of $G$ relative to $\{a_s : s > 0\}$, which, in the special case
of $\Ad(a_s)$ being diagonalizable over $\R$, is defined as a subgroup of $G$ whose Lie algebra is the sum of
(real) eigenspaces corresponding to eigenvalues of $\Ad(a_1)$ with maximum absolute value. It is always contained in the expanding horospherical subgroup, and in the special case of $F$  as in \equ{defas} clearly coincides with $H$ as in \equ{defh}.
 \smallskip
 
The next theorem is a special case of \cite[Theorem 2.8]{AGK}:
\begin{thm}\label{etds} Let $G$ be a Lie group, $\Gamma$ a discrete subgroup, $F = \{a_s : s\in\R\}$ a one-parameter subgroup which is $\Ad$-diagonalizable over $\R$, $H$ the maximally  expanding horospherical subgroup of $G$ relative to $\{a_s : s > 0\}$, and $Z$  a ${\mathcal C}^1$ $(F,H)$-transversal submanifold  of $\ggm$. Then for any $x\in \ggm$, the set 
\equ{efz}
is HAW. \end{thm}
\begin{rem}\label{haw} \rm Hyperplane absolute winning (HAW) property of subsets of Euclidean spaces has been introduced in \cite{BFKRW}, and later extended to subsets of smooth manifolds in \cite{KWe}. When the  ambient manifold is one-dimensional, this notion coincides with absolute winning as introduced by McMullen \cite{Mc}. We refer the reader to those papers, as well as to \cite[\S2.1]{AGK}, for  definitions and more information. The important aspects are that the HAW property is stable under countable intersections and implies winning in the sense of Schmidt \cite{S}, which, in its turn, implies thickness.\end{rem}

\smallskip

Applying Theorem \ref{etds} to $x = \Z^{d}\in X = \SL_{d}(\R)/\SL_{d}(\Z)$ and using Proposition \ref{exceptional}, we immediately obtain
\begin{cor}\label{thick} Let 
 $F$  be as in \equ{defas} and $H$ as in \equ{defh}, and let $\nu$ be a  norm on $\R^{d}$. Suppose that 
\eq{transversality}{
\begin{aligned}
\ \text{ the critical locus } &{\mathcal K}_\nu(1) \text{ is contained in the union of finitely many}\\ \text{  } {\mathcal C}^1\text{ compact  }&(F,H)\text{-transversal submanifolds of }X. 
\end{aligned}}
Then 
$\widehat D_\nu\subset \mr$
is hyperplane absolute winning.
\end{cor}
Note that condition \equ{transversality}  is not satisfied for the supremum norm, simply because 
the whole orbit $H\Z^{d}$ belongs to $\mathcal{K}_\infty(1)$. 
However the conclusion of Corollary \ref{thick} 
still holds for $\nu = \|\cdot\|_\infty$ due to a theorem of Davenport and Schmidt: 
it is proved in \cite{Davenport-Schmidt2} that 
$\widehat D_\infty$ contains the set 
of badly approximable systems of linear forms.
The latter was shown by Schmidt to be  winning \cite{Schmidt-BA}, and, more recently,  Broderick, Fishman and Simmons \cite{BFS} established its HAW property.


We will  
now consider two special cases. 
The first is the Euclidean norm in arbitrary dimension. 
Lattices critical with respect  to the Euclidean norm have been studied as far back as the {$17$th century in the context of sphere packings, and later in the context of positive definite quadratic forms.}
See the book of Martinet \cite{mar} for a detailed account and exhaustive references.

\begin{thm}\label{euclidean}
The critical locus ${\mathcal K}_2(1)\subset X$ corresponding to the Euclidean norm on $\R^{d}$ is contained in the union of finitely many $\SO(d)$-orbits, and each orbit  
is an $(F,H)$-transversal submanifold of $X$.
Consequently, $\widehat D_2\subset\mr$ is HAW.
\end{thm}
\begin{proof}  It follows form the work of Korkine and Zolotareff (\cite{KZ}, see also \cite[Theorem 3.4.5]{mar}) that whenever $\Lambda\in {\mathcal K}_2(1)$, any lattice in $X$ sufficiently close to $\Lambda$ is an \textit{isometric image} of $\Lambda$, that is, a lattice of the form $g\Lambda$ with $g\in 
\operatorname{SO}(d)$.
For the sake of 
making the paper self-contained we state the lemmas required to prove this and indicate how to use them.
\begin{lem}[\cite{mar}, Lemma 3.4.2]\label{martinet-1}
Let $\Lambda$ be any lattice in $\R^d$.
Then there exists a neighborhood $V$ of the identity in $\GL_d(\R)$ such that for any $g\in V$ the nonzero vectors in $g\Lambda$ of minimal Euclidean norm are images under $g$ of such minimal vectors in $\Lambda$.
\end{lem}
\begin{lem}[\cite{mar}, Lemma 3.4.4(1)] \label{martinet-2}
There is an open neighborhood $W$ of $0$ in the vector space $\operatorname{Sym}_d(\R)$   of real symmetric $d\times d$ matrices such that, for $h \in W$ with $\operatorname{tr}(h) \leq 0$ and for $g \in \GL_d(\R)$ satisfying $g^tg = I_d + h$, we have either $g \in \operatorname{O}(d)$ or $\det(g)<1$.
\end{lem}
\begin{lem}[\cite{mar}, Lemma 3.4.4(2)]\label{martinet-3}
Let $C$ be a closed cone in $\operatorname{Sym}_d(\R)$ consisting of matrices with positive trace:  \eq{trace}{\operatorname{tr}(h) >0\text{ for every nonzero }h \in C.}
Then there exists a neighborhood $W_C$ of $0$ in  $\operatorname{Sym}_d(\R)$ such that
$$h\ne 0 \text{ and }h \in W_C \cap C 
 \implies \det(I_d + h) > 1.
$$
\end{lem}
Let us now proceed with the proof of Theorem \ref{euclidean}. 
Take $\Lambda_0 \in \mathcal{K}_2(1)$, and suppose $g_k \in \SL_d(\R)\smallsetminus \operatorname{SO}(d)$ are such that $g_k$ converges to the identity as $k\to\infty$  and $g_k\Lambda_0 \in \mathcal{K}_2(1)$ for all $k$.
Define symmetric matrices $h_k$ by setting ${g_k^t}g_k = I_d + h_k$; note that $h_k\ne 0$ since $g_k\notin \operatorname{SO}(d)$.
We first claim that each $h_k$ belongs to the closed cone 
\begin{equation}\label{closed-cone}
C := \left\{ h \in \operatorname{Sym}_d(\R): \vv^th\vv \geq 0 \text{ for all } \vv\in \Lambda_0 \smallsetminus \{0\} \text{ with } \|\vv\|_2 \text{ minimal}\right\}.
\end{equation}
For this, note that for any $\vv\in\R^d$ we have $$\vv^th_k\vv = \vv^t({g_k^t}g_k - I_d)\vv  = \|g_k\vv\|^2_2 - \|\vv\|^2_2,$$ 
Since both $\Lambda_0$ and $g_k\Lambda_0$ are in $\mathcal{K}_2(1)$,  for any nonzero $\vv \in \Lambda_0$ with minimal norm this implies $\|g_k\vv\|_2^2 - \|\vv\|_2^2
\geq 0$, which proves the claim.

The next claim is that $C$
as in \eqref{closed-cone} 
satisfies 
\equ{trace}.
Indeed, for any nonzero $h\in C$  consider the suitably scaled matrix $h' = c h$ ($c>0$) which lies in $W\cap C$, $W$ being the open set in Lemma \ref{martinet-2}. 
By further decreasing $c$ we can assume that there exists $g\in V$, where $V$ is as in Lemma~\ref{martinet-1},  such that $g^tg = I_d + h'$.
Since $c >0$, it suffices to prove that $\operatorname{tr}(h') >0$.

Assume, on the contrary, that $\operatorname{tr}(h') \leq 0$. Lemma \ref{martinet-2} then says that we must have $g\in \operatorname{O}(d)$ or $\det(g)<1$.
The fact that $h$ is nonzero precludes the first alternative, and so it follows that $\det(g)<1$.
Now consider the  unimodular lattice $$\Lambda := \frac1{\det(g)^{1/d}}g\Lambda_0.$$
For any nonzero  vector $\vv \in \Lambda_0$ of minimal length we have that 
$$
\left\|\frac1{\det(g)^{1/d}}g\vv\right\|_2^2 = \frac1{\det(g)^{2/d}}\vv^tg^tg\vv = \frac1{\det(g)^{2/d}}\left(\|\vv\|_2^2 + \vv^t h' \vv \right) > \|\vv\|_2^2
$$
since $h'\in C$ and $\det(g)<1$.
Lemma \ref{martinet-1} then shows that the length of the shortest vector in $\Lambda$ must be greater than that for $\Lambda_0$.
However the fact that $\Lambda_0 \in \mathcal{K}_2(1)$ actually implies that $\Lambda_0$ is a global maximum of the function $\delta_2$ defined in \equ{delta} with $\|\cdot\| = \nu_2$,
a contradiction.

Thus our claim is proved, and 
the stage is set for applying 
Lemma \ref{martinet-3}. Indeed, we have $h_k\in C\smallsetminus \{0\}$; since $h_k \to 0$, we can assume that $h_k\in W_C$ for large enough $k$.
%
%
Hence $\det(I_d + h_k) = \det(g_k)^2>1$, contradicting the assumption that $g_k\in \SL_d(\R)$.

This argument, together with compactness of $\mathcal{K}_2(1)$, implies that the critical locus ${\mathcal K}_2(1)$ is contained in the union of finitely many $\SO(d)$-orbits.
Thus it suffices to check the transversality conditions for  $Z$ being a single 
orbit; that is, the validity of \equ{translie1} and  \equ{translie2} for $L = \SO(d)$. 
The latter is straightforward, since $\Lie (L) = \mathfrak{so}(d)$ consists of skew-symmetric matrices and hence does not contain $\Lie (F)$ for $F$  as in 
 \equ{defas}; likewise,  $\Lie (H)$ for $H$ as in 
 \equ{defh} consists of upper-triangular matrices and hence is not contained in $
 \mathfrak{so}(d)\oplus\Lie(F)$.
\ignore{It remains to conclude that the last statement of the theorem follows from the first two statements in view of Theorem \ref{etds}.}
\end{proof}

From now until the end of the paper we restrict our attention to $m=n=1$ and prove 
Theorems  \ref{winning}--\ref{dtpsieucl2}.  Recall that in this low-dimensional case we are working with  $X = \ggm$, where $G = \SL_{2}(\R)$ and $\Gamma = \SL_{2}(\Z)$, and the subgroups $F,H$ of $G$ are  one-parameter of the form 
\eq{hf}{
H = \{u_\alpha : \alpha\in\R\},\ F = \{a_s : s\in\R\}, \text{ where } u_\alpha =  \left[ {\begin{array}{cc}
   1 & \alpha \\
   0 & 1 \\
  \end{array} } \right],\ a_s = \left[ {\begin{array}{cc}
   e^{s}  & 0 \\
   0 & e^{-s} \\
  \end{array} } \right].
}
For any point $x\in X$, the left action of $G$ on $X$ induces a local diffeomorphism $G \to X$,  $g \mapsto gx$.
We identify the tangent space $T_x(X)$ with $\operatorname{Lie}(G)$ through this map.
Thus any subalgebra of $\operatorname{Lie}(G)$ defines a distribution in the tangent bundle $T(X)$.
Note that in the case when $Z\subset X$ is a one-dimensional submanifold and with $F, H$ as above, it is easy to check that the transversality conditions of Definition \ref{trans} are equivalent to the statement that at each $z \in Z$, 
\eq{tangent}{T_z(Z) 
\text{ is not contained in the distribution generated by} \operatorname{Lie}(P),}
where $P:= \left\{\left[ {\begin{array}{cc}
   a & b \\
   0 & a^{-1} \\
  \end{array} } \right]\right\}$ is the group of upper-triangular $2\times 2$ matrices.

\smallskip

Another case for which we verify  \equ{transversality} is for norms in $\R^2$ whose unit balls are not parallelograms.
Due to the nature of the argument, we have relegated the proof of the following theorem to the Appendix:
\begin{thm}\label{special-theorem-appendix}
If $\nu$ is a norm in $\R^2$ whose unit ball is not a parallelogram, the critical locus $\mathcal{K}_{\nu}(1)$ is contained in a one-dimensional $(F,H)$-transversal $\mathcal{C}^1$-submanifold of $X$.
\end{thm}

  \begin{proof}[Proof of Theorem \ref{winning} assuming Theorem \ref{special-theorem-appendix}.]  Recall that we are given an arbitrary norm on $\R^2$ and need to prove that $\widehat D_\nu\subset \R$ is  absolute winning. The latter notion, as was mentioned in Remark \ref{haw},  is a one-dimensional version of the HAW property.
Theorem \ref{special-theorem-appendix} verifies  \equ{transversality} for norms whose unit balls are not parallelograms, thus the conclusion of Theorem \ref{winning} in this case follows from 
Theorem \ref{etds} and Proposition \ref{exceptional}.

It remains to  consider  the case 
{$\nu(\x)  = \lambda\|g^{-1}\x\|_\infty$ for some $g \in \SL_2(\R)$ and $\lambda \in \R_{>0}$.  
In this case the critical locus $\mathcal{K}_{\nu}(1)$ of $\nu$ coincides with $g\mathcal{K}_{\infty}(1)$.}
Recall  that, according to \eqref{critlocus}, $\mathcal{K}_{\infty}(1)$ is equal to
the union of two compact one-dimensional manifolds, namely $$\mathcal{K}_{\infty}(1) =     H\Z^2 \cup ({\theta}H{\theta}^{-1})
\Z^2=     H\Z^2 \cup {\theta}H
\Z^2
, \text{ where }
{\theta} = \left[ {\begin{array}{cc}
   0  & -1 \\
   1 & 0 \\
  \end{array} } \right].
$$
Therefore  $\mathcal{K}_{\nu}(1)$ can be written as  $Z_1  \cup Z_2$, where 
$$Z_1:=  gH\Z^2  = (gHg^{-1})g\Z^2 \text{ and }Z_2:= g\theta H\Z^2  = \big(g 
{\theta}H(g{\theta})
^{-1}
\big)g\Z^2,$$ i.e.\ it is the union of two closed orbits of the lattice $g\Z^2$ by subgroups conjugate to $H$. 

Let us start with $Z_1$ and consider two cases:
\begin{itemize}
\item  
If  $gHg^{-1}$ is not contained in $P$, then \equ{tangent} holds for $Z = Z_1$, which implies that $Z_1$ is $(F,H)$-transversal. Thus it follows from  Theorem \ref{etds} that the set
of $\alpha\in\R$ such that 
\eq{nolimpts}{\text{there are no limit points of } \{a_{s}u_\alpha\Z^2 : s\ge 0\} \text{ in }Z_1}
is absolute winning.
\item    $gHg^{-1}\subset P$; this  happens if and only if $g = a_{s_0}h$ for some $s_0\in\R$ and $h\in H$; hence $Z_1 = a_{s_0}H\Z^2$. Clearly then \equ{nolimpts} is equivalent to the statement that 
\eq{nolimptsnew}{\text{there are no limit points of } \{a_{s}u_\alpha\Z^2 : s\ge 0\} \text{ in }H\Z^2.}
But \equ{nolimptsnew} is satisfied for any $\alpha\in\widehat D_\infty$, again  in view of  Proposition \ref{exceptional} and the description of the critical locus for the supremum norm.  Since $\widehat D_\infty$ is known to be absolute winning (as was mentioned in the introduction,   it contains the set $\BA$ which was shown to be absolute winning by McMullen \cite{Mc}), it follows that the set
of $\alpha\in\R$ satisfying \equ{nolimpts} is absolute winning in this case as well.
\end{itemize}
The argument taking care of $Z_2$ is identical, with $g$ replaced by $g\theta$. Using the intersection property of absolute winning sets, we conclude that the set
$$\widehat D_\nu = \big\{\alpha\in\R: \text{there are no limit points of } \{a_{s}u_\alpha\Z^2 : s\ge 0\} \text{ in }Z_1\cup Z_2\big\}$$
is absolute winning.
\end{proof}

\ignore{\begin{rem}\label{winningremark} \rm Using   methods of the paper \cite{KWe}, it is possible to replace the conclusion of Theorem \ref{winning} by a stronger one: namely, that the set $\widehat D_\nu$ is {\sl absolutely winning} in the sense of McMullen \cite{Mc}. Furthermore, in a recent work of the first-named author with An and Guan \cite{AGK} 
it is shown that in the conclusion of Corollary \ref{etds}, and hence also of Theorem \ref{euclidean}, thickness can be replaced by the {\sl hyperplane absolute winning (HAW)} property. The latter was defined in \cite{BFKRW} for subsets of $\R^n$ and  adapted to subsets of smooth manifolds in \cite{KWe}. Among the features of this stronger property is invariance under diffeomorphisms and stability under taking countable intersections.
\end{rem}}




\section{The targets in the upper half-plane}\label{h2}
For the remaining part of the paper we will
only consider the Euclidean norm on $\R^2$. 
To simplify notation from now on we will drop the ``Euclidean'' subscript $2$ whenever it does not cause confusion, that is, denote by $B(r)$ the Euclidean ball in $\R^2$ of radius $r$ centered at $0\in \R^2$, and by ${\mathcal K}(r)$, $r \le 1$, the subsets of  the space $X = \ggm$ of unimodular lattices in $\R^2$
given by
 \eq{k2}{{\mathcal K}(r) =  \left\lbrace \Lambda \in X: \Lambda \cap B\left(  r/{\sqrt{\Delta}}\right) = \{0\}\right\rbrace,}
where $\Delta = \sqrt{3}/2$. Then we will have 
$$D(\psi) = \{\alpha\in\R: a_s\Lambda_\alpha\notin {\mathcal K} \big(r(s)\big) \text{ whenever $s$ is large enough}\}.$$ 
Since the norm $\|\cdot\|$ is rotation-invariant, so are the sets \equ{k2} for any $r$. 
Furthermore, with the notation $K := \SO(2)$, we see that  the critical locus ${\mathcal K}(1)$ is the $K$-orbit of a single  lattice, namely the hexagonal lattice inscribed in a disk of radius $1/\sqrt{\Delta} = \sqrt{2/\sqrt{3}}$. In other words, \eq{critical}{{\mathcal K}(1) = Kg_0\Z^2,\text{ where }g_0:=\left[ {\begin{array}{cc}
   1/\sqrt{\Delta} & 1/2\sqrt{\Delta } \\
   0 & \sqrt{\Delta} \\
  \end{array} } \right] \in G.} See \cite[page 32]{ca} for a proof.

In view of the  rotational invariance of the problem it is natural to  move it to the quotient space of $G$ by $K$, that is, to the hyperbolic plane. Let $\Hyp$ denote the half-plane  {of complex numbers $z=x+iy$ with $y>0$.
We identify the tangent space $T\Hyp$ with $\Hyp \times \mathbb{C}$ and give it the Riemannian metric $\frac{dx\otimes dx + dy\otimes dy}{y^2}$. 
$T^1\Hyp$ is the set of unit tangents vectors, explicitly given as $(x+iy, \xi_1+i\xi_2)$ with $\frac{\xi_1^2+\xi_2^2}{y^2}=1$}.
{The M\"obius action of $G$ on $\Hyp$ is defined as
$$gz = \left[ {\begin{array}{cc}
   a & b \\
   c & d \\
  \end{array} } \right]z := \frac{az+b}{cz+d}
$$
and is an isometry in this metric. }
{Thus we have 
an induced left action of $G$ on $T^1\Hyp$ given by $$g(z,\xi) := 
 \left(\frac{az+b}{cz+d}, \frac{\xi}{(cz+d)^2}\right).$$
The action is in fact transitive and  (up to the subgroup of index $2$) free.}
\smallskip

In order to make use of the left $K$-invariance of ${\mathcal K}(r)$, we work with the following right actions $T^1\Hyp\curvearrowleft G$ (and also $\Hyp \curvearrowleft G$):
$$\left(z, \xi\right)\cdot g 
 :=
g^{-1}
\left(z, \xi\right),
\quad z\cdot g := g^{-1}z.
$$
We use these actions to obtain a bi-equivariant double cover $\phi:G \to T^1\Hyp$: $\phi(g)=(i,i)\cdot g$. Moreover, $\phi$ descends to a diffeomorphism, which we will also denote by $\phi$, of the left $G$-spaces $X$ and $T^1\Hyp/\Gamma$, which is a circle bundle ({away from two points}) over the  {manifold} $\Sigma := \Hyp/\Gamma.$ With some abuse of notation, let us denote by $\eta$ (resp.\ $\pi$) all the projections to quotients  by $\Gamma$ (resp.\ from tangent bundles to base spaces). We thus have the following commuting diagram:

\begin{equation}\label{diagram-chase-to-upperhalfplane}
\begin{tikzcd}
   G   \arrow[r,"\phi"] \arrow[d, "\eta"] & T^1\Hyp \arrow[d,"\eta"]  \arrow[r, "\pi"] & \Hyp \arrow[d , "\eta"] \\
  X  \arrow[r, "\phi"]  & (T^1\Hyp)/\Gamma  \arrow[r,"\pi"] & \Sigma
\end{tikzcd}
\end{equation}

Our goal now is to describe the sets $D(\psi)$
dynamically by restating Proposition \ref{danicorr} in the language of hyperbolic geometry. 
We shall identify the subsets ${\mathcal K} (r)$ of $X$ with their images under $\phi$; their rotation-invariance implies that ${\mathcal K} (r) = \pi^{-1}\Big(\pi\big({\mathcal K} (r)\big)\Big)$ for any $r$.
Furthermore,  let us put ${(z_0,\xi_0) := \phi(g_0) = (i,i)\cdot g_0,}$ where $g_0$ is as in \equ{critical}. Then \eq{z0}{z_0 := g_0^{-1}i = -\tfrac12 + i\tfrac{\sqrt{3}}2 \in \Hyp,}
and \equ{critical} can be used to describe the $\phi$-image of the critical locus ${\mathcal K} (1)$ in  $T^1\Hyp/\Gamma$ as
$$
{\mathcal K}(1) = \pi^{-1} \big(\eta(z_0)\big).
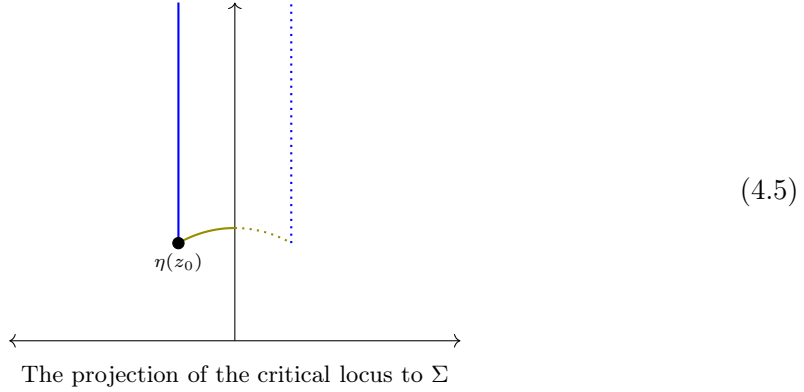
$$

\begin{equation}\label{fundamental-domain-Gamma}
\begin{tikzpicture}[scale=1.5, baseline=(current  bounding  box.center)]
\draw [black,<->] (-2,0) -- (2,0);
\draw [black, ->] (0,0) -- (0,3.00);
\draw [blue, thick] (-1/2,1.7320/2) -- (-1/2,3);
\draw [blue, thick, dotted] (1/2,1.7320/2) -- (1/2,3);
\draw [domain=0:0.5, smooth, variable=\x, olive, thick, dotted] plot ({\x},{sqrt(1-(\x)^2)});
\draw [domain=-0.5:0, smooth, variable=\x, olive, thick] plot ({\x},{sqrt(1-(\x)^2)});
\draw [fill, black] (-0.5,1.7320/2) circle [radius=0.05];
\node [below] at (-0.5,1.7320/2) {\tiny $\eta(z_0)$};
\node [below=0.2cm, align=flush center,text width=8cm]
        {\footnotesize
            The projection of the critical locus to $\Sigma$
        };
\end{tikzpicture}
\end{equation}

\smallskip
In what follows it will be useful to consider the preimage  $\eta^{-1}\big({\mathcal K} (r)\big)$ of ${\mathcal K}(r)$ in $T^1\Hyp$ as well as in $G$. 
We will use the notation $\widetilde{\mathcal K}(r)$ for both of these sets, context making clear which is in use. 
The above observations imply that \eq{z0orbit}{\pi\big(\widetilde {\mathcal K}(1)\big) = \text{ the $\Gamma$-orbit of }z_0{\text{ in }\Hyp}.}

Now take $\alpha\in\R$ and observe   that $\phi$ sends $
{u_\alpha = }  \left[ {\begin{array}{cc}
   1 & \alpha \\
   0 & 1 \\
  \end{array} } \right]$ to $(-\alpha + i, i)$. {In other words}, $$\phi(\Lambda_\alpha) = \eta\big((-\alpha + i, i)\big)$$
  lies on the closed horocycle on $T^1 \Hyp/\Gamma$ passing through $\eta(i,i)$.  Furthermore, the action of 
 ${a_s = \left[ {\begin{array}{cc}
   e^{s}  & 0 \\
   0 & e^{-s } \\
  \end{array} } \right] 
  }$ on $G$ and on $X$ 
 translates into the {(negative time direction)} geodesic flow on $T^1 \Hyp$. 
 That is, 
\eq{traj_1}{\phi(a_s\Lambda_\alpha) = \eta\big((-\alpha + e^{-2s}i, e^{-2s}i)\big).}
We have arrived at the following geometric restatement of  Proposition \ref{danicorr} for the case of Euclidean norm on $\R^2$:

\begin{prop}\label{danicorr2d}For any non-increasing continuous $\psi$, let $r(\cdot)$ be the unique function related to $\psi$ via \eq{dani2d}{r \left(\tfrac{1}{2}\ln\big(\tfrac t{\psi(t)}\big)\right) =\sqrt{t\psi(t)},}
which is the $m=n=1$ case of \equ{dani}. Then 
\eq{psi2complement}{\begin{aligned}\alpha \in D(\psi)^c & \iff a_s\Lambda_{\alpha} \in {\mathcal K}\big(r(s)\big) \text{ for an unbounded set of } s > 0\\
& \iff  -\alpha + e^{-2s}i \in \pi\left(\widetilde {\mathcal K}\big(r(s)\big)\right) \text{ for an unbounded set of }s > 0.\end{aligned}}\end{prop}

This  enables us to easily answer Questions (iii) and (iv) from the introduction for the case of the Euclidean norm on $\R^2$, and to lay a crucial groundwork for our approach to Question (v).

\begin{proof}[Proof of Theorem \ref{dtpsieuclcritical}.]  When $\psi = \psi_1$, $r(s)$ becomes the constant function $r(s)\equiv 1$.
Thus, in view of \equ{z0orbit} and \equ{psi2complement}, $\alpha \in D (\psi_1)^c$ if and only if the ray  $\{-\alpha + e^{-2s}i : s > 0\}$ hits the $\Gamma$-orbit of $z_0$ for an unbounded set of $s$.
However an elementary computation using  \equ{z0} shows that for $\begin{pmatrix}a & b\\c & d\end{pmatrix}\in\Gamma$ one has $$\RP\left( \frac{az_0 + b}{cz_0 + d}\right) = \frac{ac + bd - \frac{ad+bc}2}{c^2 - cd + d^2}\in\Q.$$ 
Therefore $\alpha\in \Q$, which implies that the trajectory \equ{traj_1} diverges in $X$, thus cannot return to a compact set infinitely many times.
\end{proof}

\begin{proof}[Proof of Theorem \ref{dtpsieucl1}.] \ignore{By equation \equ{psinudani} and the above diagram we have, \begin{equation}\label{psi2complement}\begin{split}\alpha \in D (\psi)^c & \iff a_s\Lambda_{\alpha} \in {\mathcal K} \big(r(s)\big) \text{ for an unbounded set of } s\\
& \iff -\alpha + e^{-2s}i \in \pi\big({\mathcal K} (r(s)\big) \text{ for an unbounded set of }s \end{split}\end{equation}
Note, the last condition, which is viewed in $\Hyp$, follows from the left $\SO(2)$-invariance and right $\SL_2(\Z)$-invariance of ${\mathcal K}(r)$.
\smallskip
}
Let $\psi$ be any continuous, non-increasing function  {satisfying \equ{smaller}}.
Then, in view of \equ{dani2d}, $r(s)$ is strictly less than $1$ for all large enough $s$, whence $\pi\left(\widetilde {\mathcal K} \big(r(s)\big)\right)$ is a set whose interior contains the $\Gamma$-orbit of $z_0$. To show $D(\psi)^c\neq \varnothing$ in this case, we use the simple observation that the set of real parts of $\{\gamma z_0 : \gamma \in \Gamma\}$ is dense in $\R$.
We may thus choose, inductively, a sequence $(\gamma_k) \subset \Gamma$ along with rectangular neighborhoods $U_k = A_k\times B_k$ of $\gamma_kz_0$ such that 
$\overline{A_{k+1}} \subset A_{k}$ and $U_k \subset \pi\left(\widetilde {\mathcal K}\big(r(s_k)\big)\right)$ for every $k$, where $s_k$ is defined by $e^{-2s_k}=\im(\gamma_kz_0)$. Then 
$s_k\to\infty$, {and \equ{psi2complement} shows that any element of $\bigcap A_n$ belongs to $D_2(\psi)^c$.}  
Hence the latter set is non-empty. 
In fact, a `Cantor set' type argument will show that  $D_2(\psi)^c$ is 
uncountable.
\end{proof}

The rest of the paper is devoted to proving Theorem \ref{dtpsieucl2}, which, in view of Proposition~\ref{danicorr2d}, deals with geodesics in $\Hyp{/\Gamma}$  visiting a nested sequence of sets $\pi\big( {\mathcal K}(r) \big)$, which as  $r\to 1$ converge to $\pi\big( {\mathcal K}(1) \big)= \eta(z_0)$. The goal of the remaining part of this section is to show  that the sets $\pi\big( {\mathcal K}(r) \big)$ with $r$ sufficiently close to $1$ can be efficiently approximated by small balls centered at $\eta(z_0)$:

{\begin{prop}\label{estimates}
The exist positive constants $c_0$ and $c_0'$ such that for all small enough positive $\e$
\eq{ballestimates}{B_\Hyp\big(\eta(z_0),c_0\e\big) \subset \pi\big( {\mathcal K}(1-\e) \big) \subset B_\Hyp\big(\eta(z_0),c_0'\e\big).
}
\end{prop}}

Here and hereafter by $B_\Hyp(z,\rho)$ we will mean the $\rho$-ball  centered at $z$  either in $\Hyp$ (with respect to the hyperbolic metric) or in $\Sigma$ (with respect to the induced quotient metric on $\Sigma$).

To prove this, we give a much more precise description of our shrinking targets projected to the modular surface $\Sigma$:
\begin{lem}
 Let $D$ denote the fundamental domain illustrated in diagram \eqref{fundamental-domain-Gamma}.
 In the notation of diagram \eqref{diagram-chase-to-upperhalfplane}, for any $r \le 1$ we have\begin{equation}\pi\big(\mathcal{K}(r)\big) = \eta\left(\left\lbrace z \in D: \im{z}\leq  {\Delta}/{r^2} \right\rbrace\right).
 \end{equation}
\end{lem}
\begin{equation*}\label{targets-in-domain}
\begin{tikzpicture}[scale=2.5, baseline=(current  bounding  box.center)]
\draw [black, <->] (0,0) -- (0,1.5);
\draw [blue, thick, ->] (-1/2,1.7320/2-1/2) -- (-1/2,1.5);
\draw [blue, thick, ->] (1/2,1.7320/2-1/2) -- (1/2,1.5);
\draw [domain=0:0.5, smooth, variable=\x, blue, thick] plot ({\x},{sqrt(1-(\x)^2)-1/2});
\draw [domain=-0.5:0, smooth, variable=\x, blue, thick] plot ({\x},{sqrt(1-(\x)^2)-1/2});
\draw [fill, black] (-0.5,1.7320/2-1/2) circle [radius=0.05];
\draw [red, thick, dotted, <->] (-1,0.97-1/2) -- (1,0.97-1/2); 
\node[below] at (1,1-1/2) {\tiny $y = \Delta/r^2$};
\node [below] at (-0.5,1.70/2-1/2) {\tiny $z_0$};
\node [below=0.4cm, align=flush center,text width=8cm]
        {\footnotesize
            The targets $\pi\big(\mathcal{K}(r)\big)$ are images of points in the fundamental domain below the red line.
        };
\end{tikzpicture}
\end{equation*}
\begin{proof}
We begin by showing that the first set is contained in the second. 
Choose any lattice $\Lambda \in \mathcal{K}(r)$ and, further, $g\in G$ such that $\Lambda = g\Z^2$ and such that $i\cdot g\in D$.
It suffices to show $\im{(i\cdot g)} \leq {\Delta}/{r^2}.$

In light of the $\SO(2)$-invariance of the sets $\mathcal{K}(r)$, we may as well assume 
\begin{equation}\label{upper-triangular}
g = \left[ {\begin{array}{cc}
   a & b \\
   0 & 1/a \\
  \end{array} } \right]
  \end{equation}
  is upper triangular.
The fact that $g\Z^2 \in \mathcal{K}(r)$ implies $a^2  = \left\|g\begin{pmatrix}1\\ 0\end{pmatrix}\right\|^2\geq  \dfrac{r^2}{\Delta}$, from which we can conclude $\im{(i\cdot g)}  = \dfrac1{a^2} \leq  \dfrac{\Delta}{r^2}.$

For the other containment, let $z = i\cdot g \in D$ with $\im{z} \leq {\Delta}/{r^2}$.
Again we can assume $g$ is as in \eqref{upper-triangular}, and we are left with showing that the lattice $\Lambda = g\Z^2$ belongs to $\mathcal{K}(r)$.
 The assumption $i\cdot g 
= \dfrac{\frac1a i - b}a\in D$ implies 
\begin{equation}\label{fd-condition}
    \left|\frac{b}a\right| \leq \frac{1}{2} \text{ and } \frac{b^2}{a^2} + \frac{1}{a^4} \geq 1 
\end{equation}(the second inequality follows since   $z\in D$ implies $|z| \ge 1$).
We take a non-zero integer vector $\vv = \begin{pmatrix} m\\n\end{pmatrix}$ and compute its squared norm as
\begin{equation}
\begin{split}
    (ma + nb)^2 + n^2/a^2 = m^2a^2 + 2mnab + n^2\left(b^2 + \frac{1}{a^2}\right) \\
    \geq m^2a^2 + 2mnab + n^2a^2 = a^2\left(m^2 + 2mn\frac{b}{a} + n^2\right) \\
    \geq a^2\left(m^2 - 2|mn|\left|\frac{b}{a}\right| + n^2\right) \geq a^2\left(m^2 - |mn| + n^2\right) \geq a^2.
    \end{split}
\end{equation}
In addition, we have $a^2\geq  {r^2}/{\Delta}$, since we assumed $\im{z} \leq {\Delta}/{r^2}$, hence $\|\vv\|\ge r/\Delta$, which finishes the proof.
\end{proof}

\begin{proof}[Proof of Proposition \ref{estimates}]
Since the hyperbolic distance on $\Hyp$ satisfies the identity
\begin{equation*}
\sinh\left(\frac{d_{\Hyp}(z,w)}{2}\right) = \frac{|z-w|}{2\sqrt{\operatorname{Im}(z)\operatorname{Im}(w)},}
\end{equation*}
we see that the distance from $z_0$ to the point $(-1/2+x) + i\Delta/r^2$ increases as $|x|$ increases.
Hence, when $r$ is sufficiently close to $1$ so that $\Delta/r^2 <1$, we get the following estimates for $\pi\big( {\mathcal K}(r) \big)$ by computing the distance from $z_0$ to the point $-1/2 + i\Delta/r^2$ and to the point of intersection of the unit circle with the line $y=\Delta/r^2$:
\begin{equation*}
    B_\Hyp\big(\eta(z_0),-\ln(r^2)\big) \subset \pi\big( {\mathcal K}(r) \big) \subset B_\Hyp\Big(\eta(z_0), \ln 2 + \ln\left(r^2-\sqrt{r^4-3/4}\right)\Big).
\end{equation*}
Note that $$\left.\frac{d\big(-\ln(r^2)\big)}{dr}\right\vert_{r=1} = \left.\frac{d\big( \ln 2 + \ln\left(r^2-\sqrt{r^4-3/4}\big)\right)}{dr}\right\vert_{r=1} = -2.$$

Thus we may choose any $c_0 < 2 < c_0'$ to guarantee \equ{ballestimates} for small enough positive $\varepsilon$.
\end{proof}

   \section{A zero-one law on the space of lattices}\label{sectionmau}
  
We use the following theorem of Maucourant to obtain a zero-one law in the space of lattices.
   
\begin{thm}[\cite{mau}]\label{mau1}  Let $\big(B_\Hyp(p,r_t)\big)_{t\geq 0}$ be a shrinking family of balls with radius $r_t$ in $V$, a finite volume 
hyperbolic manifold with Liouville measure 
$\mu$ on its unit tangent bundle $T^1V$.
Let $\pi$ be the projection from $T^1V$ to $V$, and let $\gamma_t$ denote the geodesic action of $\mathbb{R}$ on $T^1V$. 
Then for $\mu$-almost every (resp.\ $\mu$-almost no) $v \in T^1V$, the set 
   \begin{equation}
   \left\lbrace t\geq0 : \pi(\gamma_t v) \in  B_\Hyp(p,r_t)\right\rbrace
   \end{equation}
   is unbounded (resp.\ bounded) provided  $\int_{0}^{\infty}r_t\,dt$ diverges (resp.\ converges). $\square$ \end{thm}

We would like to restate this theorem 
according to our needs.

\begin{cor}\label{Cor-of-mau}
Let $\big(B_{\Hyp}(\eta(z_0),r_t)\big)_{t>0}$ be a family of shrinking balls in $\Hyp/\Gamma$ with respect to the quotient metric, with $z_0$ as in \equ{z0}. 
Then for Haar-almost every (resp.\ almost no) $g \in G/\Gamma$, 
\begin{equation}
    \big\lbrace t\geq 0 : a_tg \in (\pi\circ\phi)^{-1}\big( B_{\Hyp}(\eta(z_0),r_t)\big)  \big\rbrace
\end{equation}
is unbounded (resp.\ bounded) provided $\int_{0}^{\infty}r_t\,dt$ diverges (resp.\ converges).
 \end{cor}
 The difference between Theorem \ref{mau1} and Corollary \ref{Cor-of-mau} is that our targets $B_{\Hyp}(\eta(z_0),r_t)$ are centered at a branch point of the Riemann surface $\Sigma = \Hyp/\Gamma$, while the proof in \cite{mau} assumes that the surface $V$ (a finite volume quotient of $\Hyp$) admits a fundamental domain that contains a lift of the shrinking targets in its interior.
 To rectify this difficulty, we
let $\Gamma' \subset \Gamma$ be a subgroup of finite index with the
property that it acts on $\Hyp$ as a fixed-point free group of isometries, or, said in other words, that the image of $\Gamma'$ in $\PSL_2(\R)$ has no torsion\footnote{We remark that by Selberg's Lemma \cite{Se} any lattice in $G$ has a torsion-free subgroup of finite index.}. 
In this case, unlike that of $\Gamma$, the quotient map $\eta:\Hyp \to \Hyp/\Gamma'$ is not only holomorphic, but also has non-zero derivative at each point. 
This non-degeneracy ensures that $\eta$ is a local diffeomorphism, and that it induces a metric on the quotient $\Hyp/\Gamma'$, making it a hyperbolic manifold.

{So assuming $\Gamma'$ as above, we form a diagram similar to \eqref{diagram-chase-to-upperhalfplane} with $T^1\Hyp/\Gamma'$ identified with $T^1(\Hyp/\Gamma')$.}
\begin{equation}\label{unit-tangent-bundle-as-homo}
\begin{tikzcd}
  G   \arrow[r, "\phi"] \arrow[d, "\eta"] & T^1\Hyp \arrow[d, "d\eta"]  \arrow[r, "\pi"] & \Hyp \arrow[d, "\eta"] \\
  G/\Gamma' \arrow[r]  & T^1(\Hyp/\Gamma') \arrow[r] & \Hyp/\Gamma'
\end{tikzcd}
\end{equation}

{
Thus we have a map from the homogeneous space $G/\Gamma'$ to the unit tangent bundle of a hyperbolic surface, and this map is a diffeomorphism if $\Gamma'$ contains $\pm I$.}

Consider the curve $(i, i)\cdot a_{-t/2}$ in $T^1\Hyp$.
It gives the velocity vector field over a unit-speed, distance-minimizing curve, that is, the velocity field over a geodesic.
Since $g$ acts by isometries, the same is true of $(i,i)\cdot a_{-t/2}g$ and we have
\begin{equation}\label{equivariance-of-diagonal-and-geodesic}
     \phi(a_{-t/2} g)=(i,i)\cdot a_{-t/2}g = \gamma_t\big(\phi(g)\big)
\end{equation}
where $\gamma_t$ denotes the geodesic flow as in Theorem \ref{mau1}. 
Thus $\phi$ is an $\R$-equivariant map intertwining this diagonal action on $G$ and the geodesic action on $T^1\Hyp$. 
Moreover, since $\eta:\Hyp \to \Hyp/\Gamma'$ preserves the metric, this equivariance is preserved after passing to the map between $G/\Gamma'$ and $T^1(\Hyp/\Gamma').$

Lastly, consider the form $\frac{dxdyd\theta}{y^2}$ on ${T^1\Hyp}$. 
It is invariant under the right action ${T^1\Hyp}\curvearrowleft G$ and so pulls back under $\phi$ to a right invariant top form on $G$. 
Since $G$ is unimodular, this form is bi-invariant and thus descends to a left-invariant top form on $G/\Gamma'$,  a Haar measure.  
On the other hand, $\frac{dxdyd\theta}{y^2}$ gives a Liouville measure and by invariance also descends to a top form on $(T^1\Hyp)/\Gamma' \simeq T^1(\Hyp/\Gamma').$ 
By the diagram above we see that this form, pulled back to $G/\Gamma'$, is the same Haar measure. 
We now use a specific torsion-free $\Gamma'\subset \Gamma$ to prove Corollary \ref{Cor-of-mau}.

\begin{proof}[Proof of Corollary \ref{Cor-of-mau}]
{Set $\Gamma'$ in the discussion above to be the congruence subgroup $\Gamma(2)\subset \Gamma$, a torsion-free {(up to $\pm I$)} subgroup of index $6$}. As required, $\Gamma'$ acts on $\Hyp$ freely as a group of isometries and moreover contains $\pm I$. 
See Example \ref{gamma(2)} below for an example of one of its fundamental domains to keep in mind for the rest of the proof.

Combining diagrams \eqref{diagram-chase-to-upperhalfplane} and  \eqref{unit-tangent-bundle-as-homo}
gives us the following commutative diagram:

 \begin{equation}\label{big-diagram} \begin{tikzcd}
                  & G \arrow[dl] \arrow[dr] \arrow[dd] & \\
  G/\Gamma' \arrow[rr,crossing over, "\bar{\eta}" near start] \arrow [dd, "\phi" '] &    & G/\Gamma \arrow[dddd, "\pi\circ\phi"] \\
              & T^1\Hyp  \arrow[dl]   \arrow[dd]  &    \\  
   T^1(\Hyp/\Gamma')   \arrow[dd, "\pi" ']       & &                 \\
                            &   \Hyp   \arrow[dl] \arrow[dr, "\eta"] &   \\
       \Hyp/\Gamma' \arrow[rr, "\bar{\eta}"] & & \Hyp/\Gamma                                   
\end{tikzcd} \end{equation}
We apologize for the abuse of notation and hope context will remove any ambiguity. 
We may as well assume $r_t\to 0$, for otherwise ergodicity would prove the result. Note that, even though the map $\bar{\eta}$ has degree $6$, there are only two preimages of $\eta(z_0)$, each with multiplicity $3$,
see diagram \eqref{fundamental-domain-Gamma2} below.
Thus, in the bottom triangle of diagram \eqref{big-diagram}, the preimage of $B_{\Hyp}(\eta(z_0),r_t)$ under $\bar{\eta}$ is the union of two small, disjoint, hyperbolic balls which we write as $\bigcup B_{\Hyp}(p_i,r_t)$. 

\smallskip
Consider the following subsets of $G/\Gamma$ and $G/\Gamma'$ respectively:
\begin{equation*}\label{ball-01-sets}
\begin{split}
T &:= \left\{ g \in G/\Gamma: a_tg \in (\pi\circ\phi)^{-1}B_{\Hyp}\big(\eta(z_0),r_t\big) \text{ for an unbounded set of }t>0 \right\}, \\
  T' &:= \left\{ g \in G/\Gamma': a_tg \in (\pi\circ\phi)^{-1}\left(\bigcup B_{\Hyp}(p_i,r_t)\right) \text{ for an unbounded set of }t>0 \right\}.
\end{split}
\end{equation*}
In the upper triangle of \eqref{big-diagram} one can use the commutativity of the diagram to check that $\bar{\eta}^{-1}(T) = T'$. 
Since the union of two measure zero sets is of measure zero, Theorem \ref{mau1} applies equally well when the targets are a union of two balls (having the same radius for each time $t$) in $\Hyp/\Gamma'$.
Thus we have the required zero-one law for the set $T'$. Note that we have actually applied Theorem \ref{mau1} for negative times; cf.\  \eqref{equivariance-of-diagonal-and-geodesic}.
This is valid since the automorphism of the unit tangent bundle reversing the direction of tangent vectors preserves the Liouville measure. The conclusion 
now follows since 
$\bar{\eta}$, being a branched covering map, sends null sets to null sets.
\end{proof}

\begin{exmp}\label{gamma(2)}
\rm
One fundamental domain for $\Gamma(2)$ can seen as the union of six fundamental domains for $\Gamma$. The preimage $\bar{\eta}^{-1} B_{\Hyp}( \eta(z_0),r_t)$ in $\Hyp/\Gamma(2)$ is the union of two balls.

\begin{equation}\label{fundamental-domain-Gamma2}
\begin{tikzpicture}[scale=1.5, baseline=(current  bounding  box.center)]
\draw [black,<->] (-2,0) -- (2,0);

\draw [black, ->] (0-0.5,0) -- (0-0.5,3.00);

\draw [violet, thick] (-1/2 - 0.5 ,1.7320/2) -- (-1/2 -0.5,3);

\draw [domain=-0.5-0.5:0-0.5, smooth, variable=\x, brown, thick] plot ({\x},{sqrt(1-(\x-(-1-0.5))^2)});

\draw [domain=0-0.5:0.5-0.5, smooth, variable=\x, orange, thick] plot ({\x},{sqrt(1/9-(\x-(1/3-0.5))^2)});

\draw [domain=0.5-0.5:1.0-0.5, smooth, variable=\x, purple, thick] plot ({\x},{sqrt(1/9-(\x-(2/3-0.5))^2)});

\draw [domain=1-0.5:1.5-0.5, smooth, variable=\x, yellow, thick] plot ({\x},{sqrt(1-(\x-(2-0.5))^2)});

\draw [red, thick] (1.5-0.5,1.7320/2) -- (1.5-0.5,3);

\draw [teal, thick, dotted] (1/2-0.5,1.7320/6) -- (1/2-0.5,3);

\draw [domain=-0.5:1, smooth, variable=\x, violet, thick, dotted] plot ({\x-0.5},{sqrt(1-(\x)^2)});

\draw [domain=0-0.5:1.5-0.5, smooth, variable=\x, magenta, thick, dotted] plot ({\x},{sqrt(1-(\x-(1-0.5))^2)});

\draw [fill, blue] (-0.5-0.5,1.7320/2) circle [radius=0.05];

\draw [fill, black] (0.5-0.5,1.7320/2) circle [radius=0.05];

\draw [fill, blue] (1.5-0.5,1.7320/2) circle [radius=0.05];

\draw [fill, blue] (0.5-0.5,1.7320/6) circle [radius=0.05];

\node[below] at (0-0.5,0) {\tiny $(0,0)$};

\node [below=0.4cm, align=flush center,text width=12cm]
        {\footnotesize
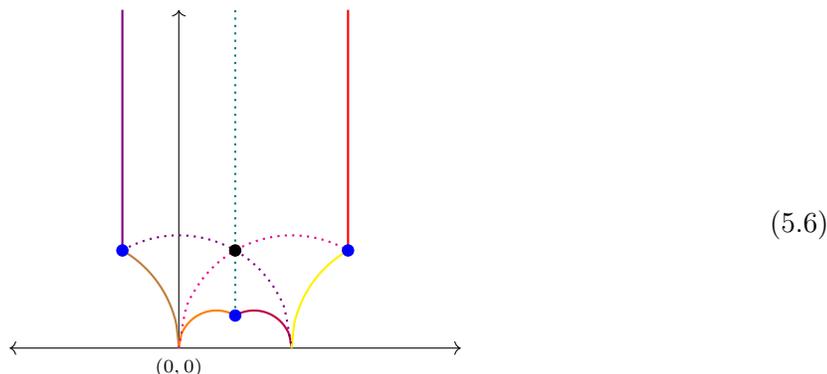
            A fundamental domain for $\Gamma(2)$ with one of the balls in its interior. The points at the center of the blue balls represent the same point in $\Hyp/\Gamma(2).$
        };
\end{tikzpicture}
\end{equation}
\end{exmp}

We can summarize the results of Sections \ref{h2} and \ref{sectionmau} as follows:

\begin{thm}\label{01-law-on-lattices}
Let $f$ be any continuous, non-decreasing function $\R_{>0} \to \R$ with $f(t) < 1$. Then the set
\begin{equation}\label{h-sets}
    T(f):= \big\lbrace \Lambda\in X : a_t\Lambda \in {\mathcal K}\big(f(t)\big) \text{ for an unbounded set of } t > 0\big\rbrace
\end{equation}
has full (resp.\ zero) Haar measure if $\int\big(1-f(t)\big)\,dt$ diverges (resp.\ converges).
\end{thm}
\begin{proof}
By ergodicity, we may as well assume $f(t)$ converges to $1$ as $t\to \infty$. In this case, Proposition $\ref{estimates}$ and Corollary $\ref{Cor-of-mau}$ give us the result.
\end{proof}

\section {From the space of lattices to a submanifold}\label{transition}

We fix the following notation, $w_y:=\left[ {\begin{array}{cc} 1 &  0 \\
                        y & 1 \end{array}}\right], u_z:=\left[ {\begin{array}{cc} 1 &  z \\
   0 & 1 \end{array}}\right]$, and $a_s = \left[ {\begin{array}{cc} e^s &  0 \\
 0 & e^{-s} \end{array}}\right]$ as before. Observe the relation between the set $T(f)$ in  $\eqref{h-sets}$ and the defining condition in \equ{psinudani}. 
If we regard $u$ as a function from $\R$ to $X$ sending $\alpha$ to $\Lambda_\alpha = u_\alpha\Z^2$, we see immediately:
\begin{lem}
 If $\psi$ is as in Theorem \ref{dtpsieucl2} and $r$ is the function defined by the property $r\left(\frac{1}{2}\ln\frac{t}{\psi(t)}\right)=\sqrt{t\psi(t)}$, which is a special case $m=n=1$  of  \equ{dani}, then \begin{equation}\label{pullback-to-dpsi}
     D_{}(\psi)^c= u^{-1}\big( T (r)\big).
\end{equation}
      \end{lem}     
In order to prove Theorem \ref{dtpsieucl2} using Theorem \ref{01-law-on-lattices}, we show that the sets $T(f)$ are invariant, in some sense, under the action of $w_y$ and $a_s$. 
This allows us to conclude that the Haar measure of $T(f)$ is locally controlled by the Lebesgue measure of $u^{-1}\big(T(f)\big)$. 
The argument  is similar to that of \cite{d} (cf.\ Proposition \ref{erg}). 
The effect of perturbing a lattice by $w_y$ or $a_s$ will be computed in terms of the  function $$\delta:\Lambda\mapsto\sqrt{\Delta}\cdot \inf\limits_{\x\in {\Lambda} \nz } \| \x\|$$  as in \equ{delta}.
With the help of this function, the sets $T(f)$ can be rewritten as 
\begin{equation*}\label{h-sets-delta}
    T(f)=\left\lbrace {\Lambda} \in X: a_s {\Lambda} \in \delta^{-1}[f(s),1] \text{ for an unbounded set of $s$}\right\rbrace.
\end{equation*}

 \begin{prop}\label{lx-estimate}  We have 
\eq{deltaestimates}{  {\delta(a_sw_y{\Lambda})}\le (1 + |y|e^{-2s}){\delta(a_s{\Lambda})}
\quad \text{and}\quad\delta(a_s{\Lambda})\le (1 + |y|e^{-2s}){\delta(a_sw_y{\Lambda})}.}
\end{prop}
                        
       \begin{proof} Note the commutation relations, $$a_sw_y=a_sw_ya_{-s}a_s= w_{ye^{-2s}}a_s.$$ 
       Now let $\vv\in \mathbb{R}^2$ be a vector in the lattice ${\Lambda}$ such that $\delta(a_s{\Lambda}) = \sqrt{\Delta}\cdot\|a_s\vv\|$. We compute \begin{equation*}
  \begin{split}
   \|a_sw_y{\vv}\| - \|a_s{\vv}\|  & = \|w_{ye^{-2s}}a_s{\vv}\| -\|a_s{\vv}\| \\
   & \leq \|w_{ye^{-2s}}-I \| \|a_s{\vv} \| \\
   & \leq |y|e^{-2s}\|a_s{\vv} \|.
  \end{split}
  \end{equation*}
This gives 
$$\sqrt{\Delta}\cdot\|a_sw_yv\| \leq (1+|y|e^{-2s})\delta(a_s{\Lambda}),$$
which implies the first inequality in \equ{deltaestimates}.
The other estimate follows similarly. \end{proof}
  
  \begin{thm}\label{01-on-pullback}
  For a continuous, non-decreasing function $f:\R_{>0} \to \R$ with $f(t) < 1$, the set $u^{-1}\big(T(f)\big)$ has full (resp.\ zero) Lebesgue measure provided $\int\big(1-f(t)\big)\,dt$ diverges (resp.\ converges).
  \end{thm}
  \begin{proof}
  Choose any $z_1 \in \R$. There exists an $\varepsilon>0$ such that the map $$\Phi:W:=(-\varepsilon,\varepsilon)^2\times (z_1-\varepsilon,z_1+\varepsilon) \to X$$ sending $(y,s,z)$ to the lattice generated by $w_ya_su_z$ is a diffeomorphism. 
  Depending on the convergence of the integral in question, we will show that $u^{-1}\big(T(f)\big) \cap (z_1-\varepsilon,z_1+\varepsilon)$ has full or zero measure. This clearly suffices to prove the theorem. \smallskip
  
\textit{Proof of the convergence case.} Assume $\int (1-f)$ converges. Define
\begin{equation*}\label{convergencesupset}
    h(t):=({\tau}_{-\varepsilon}f)(t)\cdot(1-\varepsilon e^{-2t})
\end{equation*}
where ${\tau}$ is translation, defined by $$(\tau_{-\varepsilon}f)(t) := f(t-\varepsilon).$$
The function  $h$ is still  non-decreasing, continuous, and bounded from above by $1$, 
and the set $T(h)$ still is well defined even though $h(t)$ is only defined for $t>\varepsilon$. 
Our integrability assumption on $f$ and Theorem \ref{01-law-on-lattices} imply that $T(h)$ has zero measure.
\begin{claim}
  Let $|y|,|s| < \varepsilon$. If $\Lambda \in T(f)$, then $w_ya_s\Lambda \in T(h)$.
\end{claim}  \begin{proof}
Let ${\Lambda}\in T(f)$ and $|s|<\varepsilon$. Let $(t_n)$ be a sequence witnessing ${\Lambda} \in T(f)$. Then \begin{equation*}\delta(a_{t_n-s}a_s{\Lambda})=\delta(a_{t_n}{\Lambda})\geq f(t_n)=({\tau}_sf)(t_n-s) \geq ({\tau}_{-\varepsilon}f)(t_n-s).\end{equation*} So $a_s{\Lambda} \in T({\tau}_{-\varepsilon}f)$.\smallskip

Now say ${\Lambda}\in T({\tau}_{-\varepsilon}f)$ and let $|y|<\varepsilon$. Let $t_n$ be a sequence witnessing ${\Lambda}\in T({\tau}_{-\varepsilon}f)$. 
Using Proposition \ref{lx-estimate} we see that
\begin{equation*}\delta(a_{t_n}w_y{\Lambda}) \geq \frac{\delta(a_{t_n}{\Lambda})}{1 + \varepsilon e^{-2t_n}}\geq ({\tau}_{-\varepsilon}f)(t_n)(1-\varepsilon e^{-2t_n}).\end{equation*} 
Hence $w_y{\Lambda} \in T(h)$, and the claim is proved.
\end{proof}
  
An application of the above claim to $\Lambda = \Lambda_z$ shows that $\Phi$ maps the set 
$$(-\varepsilon,\varepsilon)^2 \times \Big(u^{-1}\big(T(f)\big)\cap (z_1-\varepsilon,z_1+\varepsilon)\Big)$$ to a set of measure zero. 
The Fubini Theorem and the local equivalence of Haar measure and Lebesgue measure shows that $u^{-1}\big(T(f)\big)\cap (z_1-\varepsilon,z_1+\varepsilon)$ has Lebesgue measure zero.\smallskip

\textit{Proof of the divergence case.} The strategy is similar: we show that, in terms of the local coordinates, the union of planes above $u^{-1}\big(T(f)\big)$ contains some full measure set. 
As before, this amounts to finding some appropriate function $h$ such that $T(h)$ is full measure and such that the family of planes contains $T(h)$ as a subset. 
A naive guess based on Proposition \ref{lx-estimate} would be to use the function  {$f(\cdot)(1 + \varepsilon e^{-2(\cdot)})$}. 
However this function is not monotonic; indeed it can even be greater than $1$ on certain intervals depending on how pathological $f$ is. 
The adjustment we make below is to choose $h$ more carefully and then to throw out some measure $0$ set to ensure it is contained in our family of planes. 
\smallskip

Let $\int(1-f)$ diverge. 
Then $\int \frac{1-f}{2}$ diverges too. The function $\frac{1+f}{2}$ certainly satisfies the hypotheses of Theorem \ref{01-law-on-lattices}, 
hence $T\left(\frac{1+f}{2}\right)$ has full measure. 
The same conclusion holds if we replace $f$ by ${\tau}_{\varepsilon}f$. Another application of Theorem \ref{01-law-on-lattices} shows that $T\left(1- \varepsilon e^{-2(\cdot)}\right)$ has zero measure.
\begin{claim}\label{divergence-claim}
Let $|y|,|s| < \varepsilon$. Then $a_sw_y\left( T\left(\frac{1+{\tau}_{\varepsilon}f}{2}\right)\smallsetminus T\left(1 - \varepsilon e^{-2(\cdot)}\right)\right) \subset T(f)$.
\end{claim}

\begin{proof}
Take $$\Lambda \in  T\left(\frac{1+{\tau}_{\varepsilon}f}{2}\right)\smallsetminus T\left(1 - \varepsilon e^{-2(\cdot)}\right),$$
and let $t_n$ be a sequence witnessing 
$\Lambda \in  T\left(\frac{1+{\tau}_{\varepsilon}f}{2}\right)$. 
We can assume that the sequence satisfies \begin{equation}\label{dominating-fuction-estimate1}\frac{1+({\tau}_{\varepsilon}f)(t_n)}{2}\leq \delta(a_{t_n}\Lambda) < 1 - \varepsilon e^{-2t_n}.
\end{equation}
Since $\min(a,b)\leq \frac{a+b}{2}$, we have  \begin{equation}\label{dominating-function-estimate2}\min\left(({\tau}_{\varepsilon}f)(t) +  \varepsilon e^{-2t},1 - \varepsilon e^{-2t}\right)\leq \frac{1+({\tau}_{\varepsilon}f)(t)}{2}.
\end{equation}
Inequalities $\eqref{dominating-fuction-estimate1}$ and $\eqref{dominating-function-estimate2}$ applied to our sequence show that \begin{equation*}\label{dominating-function}({\tau}_{\varepsilon}f)(t_n) +  \varepsilon e^{-2t_n}\leq \delta(a_{t_n}\Lambda).
\end{equation*} 
Now we estimate, using Proposition \ref{lx-estimate} and the previous inequality: 
\begin{equation*}
\begin{aligned}
\delta(a_{t_n}w_y\Lambda) &\geq  \frac{\delta(a_{t_n}\Lambda)}{1 + \varepsilon e^{-2t_n}} \geq \frac{({\tau}_{\varepsilon}f)(t_n) +  \varepsilon e^{-2t_n}}{1+\varepsilon e^{-2t_n}} \\ &> \frac{({\tau}_{\varepsilon}f)(t_n)(1 +  \varepsilon e^{-2t_n})}{1+\varepsilon e^{-2t_n}}= ({\tau}_{\varepsilon}f)(t_n).
\end{aligned}\end{equation*}
Hence $w_y\Lambda\in T({\tau}_{\varepsilon}f)$. And the argument in the convergence part shows that $$a_sw_y\Lambda\in T\big({\tau}_{-\varepsilon}({\tau}_{\varepsilon} f)\big)=T(f),$$ and our claim is proved.
\end{proof}

We use this claim to show that 
$$\Phi(W) \cap \left(T\left(\frac{1+{\tau}_{\varepsilon}f}{2}\right)\smallsetminus T\left(1- \varepsilon e^{-2(\cdot)}\right)\right),
$$
a set of full measure with respect to the chart, is contained in
$$\Phi\Big((-\varepsilon,\varepsilon)^2 \times \big(u^{-1}(T(f)\big)\cap (z_1-\varepsilon,z_1+\varepsilon)\big)\Big)$$
as follows: 
take any $\Lambda$ in the first set. 
Since it is in the chart, we may write $\Lambda = w_ya_s\Lambda_z$ or equivalently,
$a_{-s}w_{-y}\Lambda = \Lambda_z
$.
Claim \ref{divergence-claim} then gives us  {exactly} what we need.

Again, by Fubini and the local equivalence of Lebesgue and Haar measure, we see that the set  $$u^{-1}\big(T(f)\big)\cap(z_1-\varepsilon,z_1+\varepsilon)$$ has full Lebesgue measure, and the divergence case is proved.\end{proof}

We now specialize to the case where $f$ is the function $r$ in Lemma \ref{pullback-to-dpsi} to give a proof of Theorem \ref{dtpsieucl2}:

\begin{cor}\label{loglaw}
Let $\psi$ be a continuous, non-increasing function such that $t\psi(t)$ is non-decreasing and $t\psi(t)<1$ for sufficiently large $t$. 
Then the Lebesgue measure of $D_{}(\psi)$ (resp.\ of $D_{}(\psi)^c$) is zero if
\eq{conditioneuclcor}{
{\sum_k\big(\psi_1(k) - \psi(k)\big)}
= \sum_k\left(\frac{1}{k} - \psi(k)\right)
= \infty \hspace{3mm}(\text{resp.} <\infty).}
\end{cor}
\begin{proof}
Applying Theorem \ref{01-on-pullback}, we see that $D(\psi)^c=u^{-1}\big(T(r)\big)$ (resp.\ $D_{}(\psi)  =u^{-1}\big(T(r)\big)^c$) has measure zero if the integral
\eq{intr}{\int \big(1-r(s)\big)\,ds} converges (resp.\ diverges).
It remains to show that the convergence of \equ{intr} is equivalent to the convergence of \equ{conditioneuclcor}. Using the definition of $r$ (cf.\ \equ{dani2d}), we compute
  \begin{equation*}\label{equivalent-convergence}\begin{split}
      \int\big(1-r(s)\big)\,ds 
             &= \int\left(1-\sqrt{t\psi(t)}\right)d\left(\frac{1}{2}\ln\frac{t}{\psi(t)}\right)\\
             &= \int\left(1-\sqrt{t\psi(t)}\right)d\left(\ln\frac{t}{\sqrt{t\psi(t)}}\right) \\
             &= \int\left(1-\sqrt{t\psi(t)}\right)d\left(\ln t -\ln\sqrt{t\psi(t)}\right). 
  \end{split}\end{equation*}
{Note that the integrals coming after the first equality are taken in the Riemann-Stieltjes sense.}

Observe that $\int\left(1-\sqrt{t\psi(t)}\right)d\left(\ln\sqrt{t\psi(t)}\right)$ is finite, and that the ratio 
$$ \frac{1-t\psi(t)}{1-\sqrt{t\psi(t)}} = 1+\sqrt{t\psi(t)}$$
is bounded between two positive constants. 
Thus the convergence of \equ{intr}
is equivalent to that of 
$\int \left(\frac{1}{t}-\psi(t)\right)\,dt$, which, in view of the monotonicity of $\psi$, is in turn equivalent to the convergence of the sum in 
\equ{conditioneuclcor}.
\end{proof}
  
\section {Appendix: Proof of Theorem \ref{special-theorem-appendix}} \label{transversal}

As promised, this section will verify Condition \equ{transversality} for  norms in $\R^2$ whose norm balls are not parallelograms.
Recall that
$G$ and $\Gamma$ denote the groups $\SL_2(\R)$ and $\SL_2(\Z)$ respectively,
$X$ denotes the space of unimodular lattices $G/\Gamma$, and $F,H\subset G$ are as in \equ{hf}.
Also recall that when $Z\subset X$ is a one-dimensional submanifold and with $F, H$ as above, the %
$(F,H)$-transversality 
of 
$Z$ is equivalent to 
\equ{tangent}.

Recall that the critical {locus} $\mathcal{K}_{\nu}(1)$ is, by definition of $\Delta_{\nu}$, the set of lattices of smallest covolume (necessarily $1$) intersecting the symmetric convex domain $B_{\nu}\left( {1}/{\sqrt{\Delta_{\nu}}}\,\right)$ trivially.
The existence of the one-dimensional submanifold containing the critical set follows from the work of Mahler whose results we make free use of.
Since his work is phrased in terms of convex domains, and since our parameterization of the hypothesized one-dimensional submanifold $Z$ comes from the boundary of such a domain, we switch from the language of norms to that of bounded symmetric convex domains in $\R^2$.

Given such a domain $B \subset \R^2$, a lattice $\Lambda$ is called $B$-\textit{admissible} if it intersects $B$ trivially.
A lattice $\Lambda$ is called $B$-\textit{critical} if it is of smallest covolume amongst all admissible lattices.
The \textit{critical determinant} $\Delta_{B}$ associated to $B$ is defined to be this smallest covolume attained. {Note that by Minkowski's convex body theorem  \cite[Theorem II, \S III.2.2]{ca}
 it is necessarily positive, and, moreover,
$\Delta_{B} \ge \frac{\area(B)}{4}$, with equality if $B$ is  a parallelogram.}
The set of all $B$-critical lattices will be called the \textit{critical locus} of $B$.

As for our previous notation, $\Delta_{\nu}$ is nothing but the critical number of the convex domain $B_{\nu}(1)$,
and  $\mathcal{K}_{\nu}(1)$ is the set of  $B_{\nu}\left( {1}/{\sqrt{\Delta_{\nu}}}\,\right)$-critical lattices.
The following theorem (from \cite[\S V.8.3]{ca}) is of fundamental importance.
\begin{thm}\label{threepairs}
Let $\Lambda$ be $B$-critical, and let $C$ be the boundary of $B$.
Then one can find three pairs of points $\pm \vp, \pm \vq, \pm \vr$ of the lattice on $C$.
Moreover these three points can be chosen such that
\begin{equation}\label{inscirbedhexagon}
\vp = \vq - \vr
\end{equation}
and any two vectors among  $\vp,\ \vq,\ \vr$ form a basis of $\Lambda.$

Conversely, if $\vp, \vq, \vr$ satisfying 
 {\eqref{inscirbedhexagon}} are on $C$,
then the lattice generated by $\vp$ and $\vq$ is ${B}$-admissible.
Furthermore no additional (excluding the six above) point of $\Lambda$ is on $C$ unless ${B}$ is a parallelogram.
\end{thm}

This theorem shows that candidates for critical lattices may be found by tracing along the boundary of ${B}$ and finding two other points satisfying equation \eqref{inscirbedhexagon}.
However such a configuration of points does not necessarily yield a critical lattice.
{Hexagons, for example, have only one critical lattice (see Lemma $13$ in \cite[V.8.3]{ca}).
One even has domains for which the critical locus is a fractal set, see 
\cite{KRS} for a discussion of this topic and concrete examples.}
However, the following notion  {due to} Mahler gives a class of domains whose critical loci are 
well behaved.

\begin{defn}
A  convex symmetric bounded domain $B$ in $\R^2$ is said to be \textsl{irreducible} if each  convex symmetric domain ${B'}\subsetneqq B$ has $\Delta_{B'} < \Delta_{B}$.
We say ${B}$ is \textsl{reducible}  if it is not irreducible, that is, if there exists $B' \subsetneqq B$ with  $\Delta_{B'}=\Delta_{B}$.
\end{defn}

 {The following was proved by Mahler in the 1940s:} 

\begin{lem}[\cite{ma2}, Lemmata 5 and 9]\label{irredcritset}
Assume ${B}$ is not a parallelogram and is irreducible.  {Then:
\begin{itemize}
\item[\rm (i)] for each $\vp \in 
\partial {B}$ 
there is exactly one ${B}$-critical lattice 
containing $\vp$;
\item[\rm (ii)] for each ${B}$-critical lattice $\Lambda$ and each $\vq,\vr \in 
\partial {B}\cap \Lambda$, all points of the line segment between $\vq,\vr$ different from $\vq,\vr$ are interior points of $B$.
\end{itemize}}
\end{lem}

\begin{lem}[\cite{ma3}, Theorem 3]\label{C1-boundary}
If ${B}$ is not a parallelogram and is irreducible, the boundary $\partial {B}$ is a $\mathcal{C}^1$-submanifold of $\R^2$.
\end{lem}
Capitalizing on these results, we have the following regularity  {statement}:
\begin{cor}\label{implicit-function}
Suppose ${B}$ is not a parallelogram and is irreducible. 
Assume further that ${B}$ is scaled so that $\Delta_{B} = 1$.
Then the locus of  ${B}$-critical lattices is a $\mathcal{C}^1$-submanifold of $X$.
\end{cor}
\begin{proof}
Let $C$ denote the boundary of ${B}$. 
In light of Lemma \ref{C1-boundary} we have a local diffeomorphism from $\R$ to $C$ given by $t\mapsto \vp(t) = \big(a(t),b(t)\big)$.
Fix a point, say $\vp:=\vp(t_0)$, and consider
$$\left\lbrace \vv: \det(\vp, \vv)=1\right\rbrace \cap C.
$$
By Lemma \ref{irredcritset} {(i)} we know that there is a unique critical lattice containing $\vp$ so that, by Theorem \ref{threepairs}, the above intersection consists of two points.
Let $\vq = \vp(t_1)$ denote 
 {one of those points}.
\begin{equation}\label{diagram-implicit}
    \begin{tikzpicture}[scale=1.6, baseline=(current  bounding  box.center)]
    \draw [<->](-1.5,0) -- (1.5,0);
    \draw [<->] (0,-0.5) -- (0,1.5);
    
    \fill [red] (0, 0) circle[radius=0.05];
    \node [right] at (0.987, 0.208) {\footnotesize $\vp$};
    \fill [red] (0.987, 0.208) circle[radius=0.05];
    
    \fill [red] (0.309,0.951) circle[radius=0.05];
    \node [above] at (0.309,0.951) {\footnotesize $\vq$};
    
    \fill [red] (-0.669,0.743) circle[radius=0.05];
    \node [above] at (-0.669 - 0.12, 0.743 + 0.02) {\footnotesize $\vq - \vp$};
    
    \draw [semithick, blue] (1,0) to [out=90, in=0] (0, 1);
    \draw [semithick, blue] (0, 1) to [out=180, in=90] (-1,0);  
    
    \draw [<->, semithick, orange] (-0.669 - 0.987,0.743 - 0.208) -- (-0.669 + 0.987 + 0.987,0.743 + 0.208 + 0.208);
    
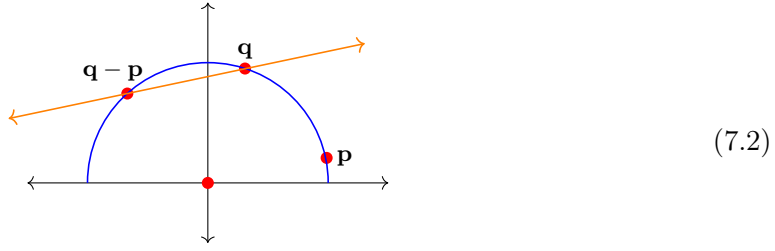
    \node [below] at (0,-0.6) {\footnotesize The intersection of $C$ and the line $\det(\vp,\cdot)=1$ consists of two points.};
    \end{tikzpicture}
\end{equation}
We claim that there is a neighborhood $W$ of $\vq$ such that for points $\vp(t)$ sufficiently close to $\vp$, the intersection
$$\left\lbrace \vv : \det\big(\vv,\vp(t)\big) =1  \right\rbrace  \cap W \cap C
$$
will consist of a unique point $\vq(t)$.
Moreover, this assignment is $\mathcal{C}^1$-differentiable.
The corollary clearly follows from this claim.

To prove the claim, we make a straightforward appeal to the implicit function theorem. 
First note that since $C$ is a $\mathcal{C}^1$-submanifold, the local immersion theorem (see \cite[Theorem $9.32$]{ru}) guarantees a neighborhood $\vq \in W$ and the existence of a non-degenerate, $\mathcal{C}^1$-differentiable function $f:W \to \R$ such that $C \cap W = f^{-1}(0)$.
Then consider the following function defined in a neighborhood of $(\vq , t_0) \in \R^3$:
$$F(x,y,t) := \big(a(t)y - b(t)x - 1, f(x,y)\big).
$$
Clearly $F(\vq, t_0) = 0$.
Moreover we have that the derivative of $F$ at $(\vq, t_0)$ is given by the matrix
\begin{equation}\label{q-implicit}
\left[{\begin{array}{ccc}
   -b(t_0) & a(t_0) & a'(t_0)b(t_1) - b'(t_0)a(t_1)\\
   \frac{\partial f}{\partial x}(\vq) & \frac{\partial f}{\partial y}(\vq) & 0\\
  \end{array}}\right].
\end{equation}

 {In view of   Lemma \ref{irredcritset}(ii), the tangent line to $C$ at $\vq$ cannot be parallel to $\vp$, or, equivalently,} 
the gradient of $f$ at $\vq$ cannot be perpendicular to $\vp$.
%
 {Hence} the leftmost $2\times 2$ minor of the matrix \eqref{q-implicit} is non-zero, and we can apply the implicit function theorem (see \cite[Theorem $9.28$]{ru}) to get a $\mathcal{C}^1$-function $\vq(t) = \big(c(t), d(t)\big)$ defined for points $t$ near $t_0$ and mapping to a neighborhood of $\vq$ such that locally,
$$
F(x,y, t) = (0, 0) \text{ if and only if $(x,y) = \vq(t).$}
$$
This proves the claim.
The function $\left[\vp(t),\vq(t)\right]$ descends to a local parameterization of the critical locus.
\end{proof}

\begin{proof}[Proof of Theorem \ref{special-theorem-appendix}.]
Let ${B}$ be the $\nu$-ball of radius $1/\sqrt{\Delta_{\nu}}$. 
In particular, we have $\Delta_{B} = 1$, so that the critical locus is a subset of $X$.
As before, let $C$ denote the boundary of ${B}$.
We first treat the case when ${B}$ is irreducible.
Since the problem is local, we fix a critical lattice $\Lambda$ and check the transversal condition at $\Lambda$.

Fix a point $\vp = r_{t_0}(\cos t_0, \sin t_0) \in \Lambda \cap C$ with $r_{t_0}>0$. 
In general every point on $C$ can be written as $\vp(t):=r_t(\cos t, \sin t)$ with $r_t>0$,
but Lemma \ref{C1-boundary} and convexity allow us to go further and say that this local parameterization of $C$ 
is actually continuously differentiable.
{Applying Corollary \ref{implicit-function}
we have a local parametrization of the critical locus in the form $M(t)\Z^2$, where
$M(t) := [\vp(t), \vr(t)]$ with {$M(t_0)\Z^2 
= \Lambda$}.}

Moreover, we can assume that $\det\big(\vp(t), \vr(t)\big) = 1$, that $\vp(t) + \vr(t) \in C$, and that $\vp(t_0)$ has the smallest angle (equal to $t_0$) modulo $[0,2\pi)$ among vectors in $\Lambda \cap C$ (see diagram \eqref{updown} below).
We write out the coordinates of $\vp$ and $\vr$ as 
\begin{equation*}
    M(t) = \left[{\begin{array}{cc}
   a(t) & c(t) \\
   b(t) & d(t) \\
  \end{array}}\right].
\end{equation*}
Our goal is to show that the differential of the curve $M(t)\Z^2$ at $t_0$, which is identified with an element 
of $\mathfrak{sl}_2(\R)$, is not contained 
in the Lie algebra $\operatorname{Lie}(P)\subset \mathfrak{sl}_2(\R)$ of upper-triangular matrices.
Under the map $g\mapsto g\Lambda$, this curve of lattices is the image of a curve in $G$ passing through the identity, namely $M(t)M(t_0)^{-1}.$
Writing $M'(t_0)$ for the component-wise derivative of $M(t)$ at $t_0$, we see that we are left with showing that the bottom left entry of 
\begin{equation*}
    M'(t_0)M(t_0)^{-1} = \left[{\begin{array}{cc}
   a'(t_0) & c'(t_0) \\
   b'(t_0) & d'(t_0) \\
  \end{array}}\right] 
  \left[{\begin{array}{cc}
   d(t_0) & -c(t_0) \\
   -b(t_0) & a(t_0) \\
  \end{array}}\right]
\end{equation*}
is never zero.
For the sake of contradiction assume it is zero. 
This can happen if and only if 
\begin{equation}\label{badcondition}
    \left(b'(t_0), d'(t_0)\right) = \lambda \big(b(t_0), d(t_0)\big) \text{ for some } \lambda \in \R.
\end{equation}
We derive a contradiction by noting the following claims:
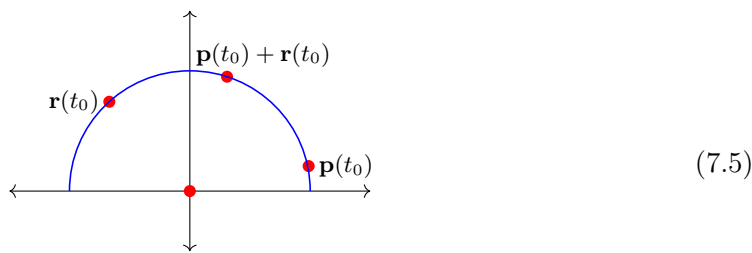
\begin{equation}\label{updown}
    \begin{tikzpicture}[scale=1.6, baseline=(current  bounding  box.center)]
    \draw [<->](-1.5,0) -- (1.5,0);
    \draw [<->] (0,-0.5) -- (0,1.5);
    
    \fill [red] (0, 0) circle[radius=0.05];
    
    \node [right] at (0.987, 0.208) {\footnotesize $\vp(t_0)$};
    \fill [red] (0.987, 0.208) circle[radius=0.05];
    
    \fill [red] (0.309,0.951) circle[radius=0.05];
    \node [above] at (0.609,0.951) {\footnotesize $\vp(t_0)+ \vr(t_0)$};
    
    \fill [red] (-0.669,0.743) circle[radius=0.05];
    \node [left] at (-0.669,0.743) {\footnotesize $\vr(t_0)$};
    
    \draw [semithick, blue] (1,0) to [out=90, in=0] (0, 1);
    \draw [semithick, blue] (0, 1) to [out=180, in=90] (-1,0);  
    
    \node [below] at (0,-0.6) {\footnotesize $b'(t_0)$ and $d'(t_0)$ must have different signs.};
    \node [below] at (0,-0.83) {\footnotesize On the other hand $b(t_0)$ and $d(t_0)$ are positive.};
    \end{tikzpicture}
\end{equation}
\begin{claim}\label{samesign}
$b(t_0)$ and $d(t_0)$ are strictly positive.
\end{claim}
\begin{proof}
First, note that $d(t_0)>0$ by choice on minimality of the  angle of $\vp(t_0)$ and by the symmetry of ${B}$.
Moreover, \eqref{badcondition} certainly cannot hold if $b(t_0)=0$ since that would also imply $b'(t_0)=0$.
Both of these cannot vanish simultaneously since $B$ is convex, and the origin is its interior point.
\end{proof}
\begin{claim}\label{differentsign}
$b'(t_0)>0$ and $d'(t_0)<0$.
\end{claim}
\begin{proof}
We first show that $b'(t_0)>0$. 
 {If not, then the slope of the tangent line to $C$ at $\vp(t_0)$ is non-negative; hence, by the convexity of $B$, the curve
\eq{curve}{\{\vp(t) : t_0 \le t \le \pi\}} is contained in the half-plane $\{(x,y) : y \le b(t_0)\}$. On the other hand, the point $\vp(t_0)+ \vr(t_0)$ lies on the curve \equ{curve}, and its $y$-coordinate is equal to
$b(t_0)+d(t_0)$, which is strictly larger that $b(t_0)$ by  Claim \ref{samesign}, a contradiction.}

\begin{equation}
    \begin{tikzpicture}[scale=1.6, baseline=(current  bounding  box.center)]
    \draw [<->](-1.5,0) -- (1.5,0);
    \draw [<->] (0,-0.5) -- (0,1.5);
    
    \fill [red] (0, 0) circle[radius=0.05];
    
    \fill [red] (0.707, 0.707) circle[radius=0.05];
    \draw [->, semithick, orange] (0,0) -- (0.707, 0.707);
    \node [below] at (0.707, 0.707) {\footnotesize $\vp$};
    
    \fill [red] (-0.259, 0.966) circle[radius=0.05];
    \draw [->, semithick, orange] (0,0) -- (-0.259, 0.966);
    \node [above] at (-0.259, 0.966) {\footnotesize $\vr + \vp$};
    
    \fill [red] (-0.966, 0.259) circle[radius=0.05];
    \draw [->, semithick, orange] (0,0) -- (-0.966, 0.259);
    \node [left] at (-0.966, 0.259) {\footnotesize $\vr$};
    
    \draw [dotted, semithick, orange] (0.707, 0.707) -- (-0.259, 0.966);
    \draw [dotted, semithick, orange] (-0.259, 0.966) -- (-0.966, 0.259);
    
    \draw [<->, blue] (1.25,0.707) -- (-1.25,0.707);
    
    \node [below] at (0,-0.6) {\footnotesize $b'(t_0)$ cannot be less than or equal to $0$.};
    \end{tikzpicture}
\end{equation}

 {The inequality $d'(t_0) < 0$ is  proved by the above argument applied to the tangent line to $C$ at $\vr(t_0)$.}
\end{proof}

Now the above two claims and condition \eqref{badcondition} are incompatible, and we have reached our contradiction. 
This shows that the critical locus of ${B}$ is transversal to the distribution generated by $\operatorname{Lie}(P)$
as required.

In order to generalize to the case when $B$ is reducible, we use Mahler's important result that each convex, bounded, symmetric domain ${B}$ contains an irreducible ${B'}$ which
has the same critical determinant, that is $\Delta_{B} = \Delta_{B'}$ (see \cite[Theorem $1$]{ma2}).
We claim that since $B$ is not a parallelogram, this irreducible $B'$ cannot be a parallelogram either.
For in this case
$$\Delta_{B'} = \frac{\area(B')}{4} < \frac{\area(B)}{4} \leq \Delta_B,$$ which is a contradiction.

Finally, the containment $B' \subset B$ and the equality $\Delta_{B'} = \Delta_B$ show that the $B$-critical locus is contained in the $B'$-critical locus. 
This completes the proof.
\end{proof}

\end{document}